\documentclass{amsart}
\usepackage{amsmath,amssymb,amsthm}
\usepackage{fancybox,ascmac}
\usepackage{bbm}
\usepackage{amsmath}
\usepackage{bbm}
\usepackage[dvips]{graphicx}
\usepackage{color}
\usepackage{comment}

\newtheorem{dfn}{Definition}[section]
\newtheorem{thm}[dfn]{Theorem}
\newtheorem{lem}[dfn]{Lemma}
\newtheorem{cor}[dfn]{Corollary}
\newtheorem{rem}[dfn]{Remark}

\newtheorem{asmp}[dfn]{Assumption}
\newtheorem{open}[dfn]{Open problem}
\newtheorem{prop}[dfn]{Proposition}\makeatletter

\newcommand{\red}[1]{{\color{red}#1}}

\newcommand{\ca}{{\rm Cap}}
\newcommand{\Z}{{\mathbb Z}}

\newcommand{\N}{{\mathbb N}}

\newcommand{\G}{{\mathcal G}}

\numberwithin{equation}{section}

\usepackage{comment}
\usepackage[utf8]{inputenc}

\title[Fluctuation of the structure of random walk ranges]{Phase transition on the fluctuation of the structure of random walk ranges}

\author{Arka Adhikari}
\address{Arka Adhikari\hfill\break
Department of Mathematics, University of Maryland- College Park, College Park, MD, USA}
\email{arkaa@umd.edu}

\author{Izumi Okada}
\address{Izumi Okada\hfill\break
Department of Mathematics, University of Tokyo, 3-8-1 Komaba, Meguro-ku, Tokyo 153-0041, Japan.}
\email{iokada@ms.u-tokyo.ac.jp}

\thanks{Research supported in part by JSPS KAKENHI Grant-in-Aid  for Early-Career Scientists (No.~JP24K16931) (I.O.).}

\begin{document}
\maketitle
\begin{abstract}
We investigate fluctuation phenomena for the graph distance and the number of cut points associated with random media arising from the range of a random walk. Our results demonstrate a sequence of dimension-dependent phase transitions in the scaling behavior of these fluctuations, leading to qualitatively different regimes, with a distinct phase transition in dimension 6. In particular, we remark that convergence in dimension 6 occurs with a non-standard rescaling.  
\end{abstract}

\section{Random walk on random walk range}
Consider the simple random walk $(S_n)_{n=0}^\infty$ on $\Z^d$ for $d\ge 4$. 
We define the graph of the random walk range as follows: 
\begin{align*}
    &V(\G):=\{S_m: m\ge 0\}, \quad
    E(\G):=\{\{S_m,S_{m+1}\}: m\ge 0 \}.
\end{align*}
Originally, Lawler and Burdzy \cite{BL90} discussed the exponents for the fluctuations of the random walk range and proposed many open problems. 
Croydon and Shiraishi \cite{CroydonShiraishi2023, Shiraishi2010, Shiraishi2012} estimated the graph distance, the effective resistance of $(V(\G), E(\G))$  and the number of cut points of the random walk range between $S_0$ and $S_n$ in $d=4$, and showed that their typical orders are $n/(\log n)^{1/2+o(1)}$. 
In addition, Croydon \cite{Croydon2009b} did similar work in $d\ge5$ and showed that their typical order is linear. In \cite{Shiraishi2018}, he discussed  the loop-erased random walk and other related concepts for the graph of the random walk range in $d=2,3$.

As our motivation, random walks in random media have attracted considerable interest as mathematical models for transport phenomena in disordered environments. Typical examples include random conductance models, random walks in random environments, and random walks on percolation clusters. One of the central themes in this area is to understand how the geometry and randomness of the underlying medium influence the long-time behavior of the walk (e.g., see \cite{Croydon2018, CroydonHambly2008}). For example, for supercritical percolation clusters on 
$\Z^d$ with $d \ge 2$, the walk behaves diffusively, in close analogy with the simple random walk on $\Z^d$ (e.g., see \cite{Barlow2004}). The situation is markedly different at criticality. In highly irregular media, such as large critical percolation clusters, it is widely believed that the random walk exhibits subdiffusive behavior (e.g., \cite{Kesten1986}). 
 In addition, \cite{Croydon2018, CroydonHamblyKumagai2017} deal with  time changes of processes on resistance spaces, which cover the case when the resistance has a non-trivial scaling limit. We refer the reader to \cite{Kumagai2010} for further references on these topics.

Independent of the study of random walks on random media, understanding the graph distance in random media is of intrinsic interest. The graph distance captures the geometry induced by the random environment itself and provides a natural notion of distance that reflects how connectivity and paths are shaped by randomness. 

The behavior of graph distance plays a central role in a variety of problems, including volume growth, metric scaling, and universality phenomena in random geometry. In particular, in light of recent developments,  it is conjectured that graph distances in several models of random media exhibit non-trivial scaling behavior, potentially related to general universality phenomena. These considerations provide strong motivation for a detailed analysis of graph distances in random media, which is the focus of the present work.

Motivated by these considerations, we aim to understand the fluctuation behavior of the graph distance in random media. In this paper, we take the range of a random walk as a concrete example of a random medium. According to the results of Croydon \cite{Croydon2009b}, in $d\ge 5$ the leading-order behavior is linear regardless of the dimension and does not possess a phase transition, which naturally raises the question of how fluctuations behave around this leading term. 

Our analysis reveals that the fluctuation behavior depends sensitively on the dimension. In particular, we find distinct regimes in dimensions four, five, six, and seven and higher. A major phase transition occurs at dimension six: for $d\ge 6$, the fluctuations, when normalized by the square root of the variance, converge in distribution to a Gaussian random variable, whereas in dimensions four and five, the same normalization leads to convergence of the fluctuation to zero. We also note that the variance of the normal random variable in $d=6$ that appears in the central limit theorem is non-standard in the sense that it is actually not equal to the variance of the original random variable. 


We will give the definitions necessary to state the main results. 
We consider the graph $\G[a,b]$ as the set of edges constructed by the random walk trace: 
\begin{align*}
    E(\G[a,b]):=\{\{S_m,S_{m+1}\}: a\le m \le b-1  \}.
\end{align*}
The graph distance between two points, as is standard, counts the number of edges in the shortest path connecting them in the graph. In addition, we define the number of cut points on $\mathcal{G}[1,n]$ by
\begin{align*}
    \sum_{i=2}^n 
    1\{S[1,i-1 ]\cap  S[i,n] = \emptyset\}. 
\end{align*}

We will now define the notation used to represent the graph distance and the number of cut points. Let $X^1_{[a,b]}(x,y):= d_{\mathcal{G}_{[a,b]}}(x,y)$ and  
$X^2_{[a,b]}(i,j)$ be the number of cut points on  $\mathcal{G}_{[a,b]}$ between the time $i$ and $j$.
A trivial relationship among these variables is $ X^1_{[a,b]}(S_i,S_j) \ge X^2_{[a,b]}(i,j)$. We sometimes write $X^i[a,b]$ as $X^i_{[a,b]}(S_a,S_b)$ 
and $X^1_n=X^1_{[0,n]}(S_0,S_n)$ and $X^2[a,b]$ as $X^2_{[a,b]}(a,b)$ 
and $X^2_n=X^2_{[1,n]}(1,n)$. 
Since the following proof is the same, we omit the symbol ``$i$", wherever it appears. 
To ease the presentation, we omit hereafter the integer-part symbol $\lceil \cdot \rceil$. 
If there exists $c\in (0,\infty)$ 
such that $c a_n \le b_n \le c^{-1}a_n$, 
we write $a_n \asymp b_n$. 
In addition, if $ a_n /b_n \to 1$ as $n\to \infty$,  we write $a_n \sim b_n$. If $ \log a_n / \log b_n \to 1$ as $n\to \infty$,  we write $a_n \approx b_n$.  
$a_n\gg b_n$ means $a_n/b_n \to \infty$ as $n\to \infty$ and $a_n\ll b_n$ means $b_n/a_n \to \infty$ as $n\to \infty$. 
We write $f (n) \lesssim g(n)$ if there exists a (deterministic) constant $c>0$ such that $f (n) \le cg(n)$  for all $n$, and $f (n) \gtrsim g(n)$ if $g(n) \lesssim f (n)$. 

We now present our main results on the fluctuation of the graph distance and the number of cut points. 
\begin{prop}\label{prop1}
Consider $i=1,2$. 
It holds that, as $n\to \infty$, 
\begin{equation*}
\mathrm{Var}(X_n^i) 
\quad 
\begin{cases}
\sim \alpha^i_d n & \text{for } d\ge 7, \\
\asymp n \log n & \text{for } d=6,\\
\asymp n^{3/2} & \text{for } d=5,\\
=  \frac{n^2}{(\log n)^{2\pm o(1)}} & \text{for } d=4,
\end{cases}
\end{equation*}
for some constant $\alpha^i_d >0$ with $i=1,2$. 
\end{prop}

\begin{thm}\label{m1}
Consider $i=1,2$. 
For $d\ge 7$, as $n\to \infty$, 
\begin{align*}
    \frac{X_n^i-\mathbb{E}X_n^i}{\sqrt{ \mathrm{Var}(X_n^i) }}
    \stackrel{d}{\rightarrow} \mathcal{N}(0,1)
\end{align*}
and for $d=6$, there is a deterministic function $\phi (n) \asymp n \log n$ such that as $n\to \infty$, 
\begin{align*}
    \frac{X_{n}^i-\mathbb{E}X_{n}^i}{\sqrt{ \phi(n) }}
    \stackrel{d}{\rightarrow} \mathcal{N}(0,1).   
\end{align*}
For $d=4,5$, as $n\to \infty$, 
\begin{align*}
    \frac{X_n^i-\mathbb{E}X_n^i}{\sqrt{ \mathrm{Var}(X_n^i) }}
    \stackrel{d}{\rightarrow} 0. 
\end{align*}
\end{thm}

\begin{open}
    Does a strong law of large numbers hold for $X_n$? (In \cite{CroydonShiraishi2023}, a weak law has been proved on $(V(\G), E(\G))$. In \cite[Remark 2.2.2]{Shiraishi2012}, the upper bound for a strong law for the effective resistance was essentially shown.)
One of the difficulties in the proof is that $X_n$ does not satisfy monotonicity with respect to $n$. Even if this obstacle could be overcome, especially in the case $d=4$, the variance is not much smaller than that of $\mathbb{E}[X_n]^2$ (they differ only by a factor of $``\log n"$ in $d=4$), so it seems unlikely that the result can be proved using standard concentration inequalities.
\end{open}

We emphasize that the theorem above on the fluctuations of the graph distance and the number of cut points demonstrates a key phase transition in dimension 6. While in dimensions strictly greater than 6, one has the standard central limit theorem, this fails in dimensions less than 6. The standard scaling for the fluctuation leads to a trivial random variable in dimensions 4 and 5. In addition, dimension 6 presents a subtle, but deep change in the behavior of the fluctuation; instead of fluctuations on the scale of $\sqrt{n}$, as in higher dimensions, the fluctuation will be of the scale $\sqrt{n \log n}$. Moreover, the precise order of the fluctuation in dimension 6 is not given by the square root of the variance, but differs from the variance by a constant. Indeed, by looking at the proof, one can intuit that the main contribution to the fluctuation in dimensions higher than seven arises primarily from local fluctuations in the random walk range. In dimension 6, the proof demonstrates that the primary source of the fluctuation arises from interactions at intermediate scales of order $\sqrt{n}$. As such, in dimension 6, we observe the first demonstration of the effect of long-range interactions for the graph distance of the random walk range.  
We remark that the source of this new phase transition is due to the effect of the heavy tailed nature of the long range interactions of the random walk that can cause significant effects on the graph distance. 

One of the key ingredients of our proof is to quantify the effect of the long-range interactions of the random walk range by a decomposition. On the highest level, one can measure the effect of the interaction between the first and second halves of the random walk by splitting the trajectory $S[0,n]$ into two subpaths $S[0,n/2]$ and $S[n/2,n]$. Once this is done, one can compare the corresponding graph distances by studying the variable
\[
X[0,n/2] + X[n/2,n] - X[0,n],
\]
that naturally arises as an error term. 


By iteratively applying this decomposition procedure, we obtain a representation consisting of a sum of finely decomposed random variables that behave approximately as independent and identically distributed terms, together with accumulated error terms. The key question here is to determine which of the terms contributes the most to the variance, and on which scale they operate. In dimensions $d \ge 7$, one can very finely decompose the random variables and show that local contributions will contribute the most to the fluctuations; since the local contributions are independent of each other, this will lead to a Gaussian fluctuation. In dimensions 4 and 5, one will observe that the variance almost entirely comes from the error term at the largest scale; namely, it comes from considering $X[0,n/2] + X[n/2,n] - X[0,n]$.

In dimension 6, there becomes a key interplay between the contributions of the error terms and the sum of the independent terms. Indeed, fully decomposing the  random walk will not lead to a correct understanding of the fluctuation; rather, the most important contributions to the fluctuation in dimension 6 come from intermediate scales of order $\sqrt{n}$. Due to the heavy tails, the probability that the error term $X[0,n/2] + X[n/2] - X[0,n]$ will make significant contributions to the fluctuation is around $n^{-1}$. At the intermediate scale of size $\sqrt{n}$, we will have $\sqrt{n}$ terms that have a probability $(\sqrt{n})^{-1}$ of contributing non-trivially. Thus, the intermediate scales of size $\sqrt{n}$ are the first scales that start to contribute with high probability to the fluctuation. The total fluctuation of $n \log n$ comes from accumulating all the interactions from scale $\sqrt{n}$ or lower and carefully taking into account their interactions. This is in contrast to the naive computation of the variance, which accumulates some level of contributions from all scales, even those that are not seen in the fluctuation. 


Many of the key challenges in determining the fluctuations in the graph distance or the number of cut points require a deep understanding of the nature of the error terms that arise in our analysis. In contrast to the study of many other variables for the random walk, we remark that the number of cut points or the graph distance is a purely discrete quantity. That is, one cannot understand the graph distance by considering the scaling limit of the random walk to the Brownian motion, and then analyzing the appropriate statistics of the Brownian motion. As a result, there is a lack of closed formulas that can be used to analyze the error terms. 

Another aspect of the analysis of the error terms is that they are very heavy tailed. This occurs because large values of this error term are governed by events that, while rather unlikely, involve very long-range dependencies in the random walk. This behavior of the error terms causes two fundamental problems for our analysis. First, long-range dependencies cause the error terms at different scales to become correlated. Secondly, one now has to take into account a balance between the rarity of a given interaction and the effect of a given error term to determine whether it contributes to the scale of our fluctuations.

Solving these problems requires a very precise analysis of the covariance structure of the error terms. In particular, many previous estimates we found in the literature were insufficient for our purposes. In higher dimensions, we sharpen a result of \cite[Lemma 1.1]{BSZ03} to identify the optimal order of fluctuation of the given error terms and remove an important $O(\log n)$ factor. A delicate analysis on the correlation of cut times and long-range intersections of the random walk allows us to determine the covariance structures of the cross terms in dimensions $d \ge 5$ (observe the proofs of Lemma \ref{lem:uprbndvar} and Lemma \ref{lem:Var5d}). In dimension 4, we  used a link between intersection properties and the capacity of the random walk, combined with a multiscale analysis in a series of  lemmas, in order to decouple the interaction between two subpaths of a random walk and the graph distance on either subpath and obtain optimal estimates on the variance of the cross terms in Section \ref{sec:dim4}. Finally, in the proof of Theorem \ref{m1}, we performed the careful balancing act that allowed us to exactly determine the correct order of fluctuation;  this involved showing that, due to the heavy tails, the highest scales occur too rarely, and then carefully analyzing the correlations between the cross terms at intermediate scales. Our proof reveals that the intermediate scales of order $\sqrt{n}$ determined the main contribution to the fluctuation in six dimensions.

In addition to our analysis of the graph distance, we present, for independent interest, an analysis of quantities for the random walk range via our decomposition strategy, whose error terms exhibit the same heavy tailed behavior as the graph distance or the number of cut points. In this more general analysis, we do not make any assumptions about the correlations of the cross term. As a result, it will not be possible to obtain the same sharp results as one could obtain for the graph distance; however, our analysis demonstrates that heavy tailed behavior, by itself, is enough to demonstrate a type of phase transition similar to those shown for the graph distance. Indeed, we believe that these more general results would be helpful for analyzing other random variables of interest, and, furthermore, demonstrate the key role of the size of the heavy tails in the cross terms in creating a phase transition  for the fluctuations. 

\begin{prop}\label{prop1+}
Let $X^3_{[a,b]}$ be any collection positive random variables that can be written as a measurable function, $\mathcal{F}(S[a,b])$,  with respect to the random walk range $S[a,b]$ (with $X^3_n:= X^3_{[0,n]}$). Here, we treat $S[a,b]$ as a subgraph of $\mathbb{Z}^d$. We further assume that $\mathcal{F}$ is translation invariant and can be bounded from above by the volume. Define $E_n:= X^3_n + X^3_{[n,2n]} - X^3_{[0,2n]}$ and assume further that $E_n$ satisfies the bounds
$$
\mathbb{P}( E_n \ge l) \asymp l^{-(d/2 -2)}, 
\quad d \ge 5,
$$
and $E_n$ is always positive, that is, $\mathcal{F}$ is subadditive. 
For $d=4$, the right-hand side above does not meaningfully characterize heavy tailed behavior, so the natural analogs of the heavy tailed assumption are those found in Assumption \ref{asmp:weaksmpheavy}.

    Then, for $d\ge 7$ and $\alpha^3_d>0$, as $n\to \infty$, 
    \begin{align*}
     \mathrm{Var}(X_n^3) \sim \alpha^3_d  n. 
\end{align*}
For $d=6$, 
    \begin{align*}
    \limsup_{n\to \infty} \frac{\mathrm{Var}(X_n^3)}{n \log n} > 0, \quad 
    \limsup_{n\to \infty} \frac{\mathrm{Var}(X_n^3)}{n (\log n)^{2}} < \infty. 
\end{align*}
For $d=5$, 
    \begin{align*}
    \mathrm{Var}(X_n^3)\approx  n^{3/2}. 
\end{align*}
For $d=4$ and $\epsilon>0$, 

\begin{align*}
    \limsup_{n\to \infty} \frac{\mathrm{Var}(X_n^3)}{n^2 /(\log n)^{2+\epsilon}} > 0, \quad 
    \limsup_{n\to \infty} \frac{\mathrm{Var}(X_n^3)}{n^2/ (\log n)^{1-\epsilon}} < \infty. 
\end{align*}


\end{prop}

\begin{thm}\label{m1+}
Let $X_n^3$ satisfy the assumptions outlined in the previous lemma. 
For $d\ge 7$, as $n\to \infty$, 
\begin{align*}
    \frac{X_n^3-\mathbb{E}X_n^3}{\sqrt{ \mathrm{Var}(X_n^3) }}
    \stackrel{d}{\rightarrow} \mathcal{N}(0,1)
\end{align*}
for $d=6$, there is a subsequence $(n_k)$ and $\phi (n) \ge c n(\log n)$ for some $c>0$ such that as $n\to \infty$, 
\begin{align*}
    \frac{X_{n_k}^3-\mathbb{E}X_{n_k}^3}{\sqrt{ \phi(n_k) }}
    \stackrel{d}{\rightarrow} \mathcal{N}(0,1).   
\end{align*}
For $d=5$, as $n\to \infty$, 
\begin{align*}
    \frac{X_n^3-\mathbb{E}X_n^3}{\sqrt{ \mathrm{Var}(X_n^3) }}
    \stackrel{d}{\rightarrow}0, 
\end{align*}
and for $d=4$, there is a subsequence $(n_k)$ such that as $n\to \infty$, 
\begin{align*}
    \frac{X_{n_k}^3-\mathbb{E}X_{n_k}^3}{\sqrt{ \mathrm{Var}(X_{n_k}^3) }}
    \stackrel{d}{\rightarrow} 0.  
\end{align*}
\end{thm}

\section{Preparations in $d\ge 5$}

In this section, we will establish various estimates that will allow us to understand the cross terms that appear in the course of our analysis.
We will give the definitions here. 
Consider the error term 
$$
E_n:= X^i[0,n]  + X^i[n,2n]  - X^i[0,2n].
$$
In general, we may write
\begin{equation*}
\begin{aligned}
E_n^{(k,l)}(i):=& X^i\bigg[  \frac{2l}{2^{k+1}}n, \frac{2l+1}{2^{k+1}}n \bigg] + X^i\bigg[ \frac{2l+1}{2^{k+1}}n, \frac{2l+2}{2^{k+1}}n \bigg]
-X^i\bigg[ \frac{l}{2^{k}}n, \frac{l+1}{2^{k}}n \bigg]. 
\end{aligned}
\end{equation*}
Since the following proofs do not depend on the quantity of interest, we omit the symbol ``$i$", wherever it appears.

In the next lemma,  we improve the result of \cite[Lemma 1.1]{BSZ03}.  
We will use the following asymptotics on the tail probabilities of $E_n$ in the following sections.  
\begin{lem}
Let $d\ge 5$. 
There exist some constants $0<c,C<\infty$ such that
\begin{equation}\label{upper+}
\mathbb{P}(E_n \ge l) \le C l^{-d/2+2}, \quad  \forall l, 
\end{equation}
and
\begin{equation}\label{lower+}
\mathbb{P}(E_n  \ge l) \ge c l^{-d/2+2},  \quad  an \le \forall l\le bn.  
\end{equation}
\end{lem}

\begin{rem}
    By the same proof, we also have for $l\le \max(m,n)$, 
    \begin{align*}
        \mathbb{P}(X^i[0,n]  + X^i[n,n+m]  - X^i[0,n+m] \ge l)
        \le C l^{-d/2+2}. 
    \end{align*}
\end{rem}

\begin{proof}[Proof of \eqref{upper+}]
For $I \subset  \Z^d$, let
\begin{align*}
    N(I)
    := \sum_{n\in I} 
    1\{S(-\infty,n-1 ]\cap  S[n, \infty) = \emptyset\},
\end{align*}
for $I \subset [a,b] \subset \Z^d$,  
\begin{align*}
    N(I,[a,b])
    := \sum_{n\in I} 
    1\{S[a,n-1 ]\cap  S[n, b] = \emptyset\}
\end{align*}
and $T^1$ (resp. $T^{-1}$) be the first global cut point after (resp. before) the origin (see the details in \cite[Lemma 1.1]{BSZ03}). 
Note that $T^1$ has the same distribution (up to a sign change) as  $T^{-1}$ and 
\begin{align*}
    \mathbb{P}(E_n \ge l ) 
    \le& \mathbb{P}( X_{[1,2n]}(n+1,T^1(\theta^{n} \omega))\ge l/2-1) + 
    \mathbb{P}( X_{[1,2n]}(n,T^{-1}(\theta^{n-1} \omega))\ge l/2-1)\\
    \le & \mathbb{P}(T^1 \ge l/2-1)+ \mathbb{P}(T^{-1} \ge l/2-1),
\end{align*}
then it suffices to show the result for $T^1$. In the equation above, we used the notation of \cite{BSZ03}, where $\omega$ is a  variable in the state space that represents the sequence of steps taken by the random walk. $\theta^n \omega$ is the translation of the walk by $n$ time steps; namely, the new walk $\hat{S}$ given by $\theta^n \omega$ satisfies $\hat{S}_i = S_{i - l}$.


We now fix a parameter $J$, whose value will be set sufficiently large later. We then define $L= l/(3J)$ and set $k_j := Lj$ for $0 \le j \le 3J$. 
If we let 
\begin{align*}
    &a_1 := \sum_{0\le j\le 3J-1} 
    \mathbb{P}(S(-\infty,k_j-1 ]\cap  S[k_{j+1}, \infty) \neq \emptyset ), \\
    &a_2 := \mathbb{P} (\cap_{j=0}^{3J-1} \{N([1, l]) = 0, S(-\infty,k_j-1 ]\cap  S[k_{j+1}, \infty) = \emptyset\}), 
\end{align*}
we have
\begin{align}\label{cc*}
\mathbb{P}(T^1 > l) \leq a_1 + a_2.
\end{align}
By the same proof in \cite[Lemma 1.1]{BSZ03} (which essentially computes the probability of having a large loop), we can obtain 
\begin{align}\label{cc**}
a_1 \leq C(J) l^{-\frac{d}{2}+2}.
\end{align}

Note that if $S(-\infty, k_{3j}-1] \cap S[k_{3j+1},\infty) = \emptyset$, $S(-\infty, k_{3j+2}-1] \cap S[k_{3j+3},\infty) = \emptyset$ and $N([1, l]) = 0$, then 
$N([k_{3j+1},k_{3j+2}-1], [k_{3j},k_{3j+3}-1])=0$.  To see the reason for this, note that $N([1,l]) = 0$ means that no point in the interval $[k_{3j+1}, k_{3j+2} -1]$ can be a cut point; thus, for $t \in [k_{3j+1},k_{3j+2} -1]$, we must have that $S(-\infty, t-1] \cap S[t, \infty) \ne \emptyset $. However, since $S(-\infty, k_{3j }-1] \cap S[k_{3j+1},\infty)=\emptyset$ and $S(-\infty,k_{3j+2} -1] \cap S[k_{3j+3}, \infty) = \emptyset$, $S[k_{3j},t-1] \cap S[t, k_{3j+3} -1]\ne \emptyset$ holds instead. 
Hence, 
\begin{align*}
    &a_2 \le 
    \mathbb{P} \bigg( \bigcap_{j=1}^J 
    \{ N([k_{3j+1},k_{3j+2}-1], [k_{3j},k_{3j+3}-1])=0\} \bigg). 
\end{align*}
For $0 \leq j < J$ and $0<\alpha<1$, we have:
\begin{align*}
A_1^{j,\alpha} 
= &\Big\{ 
   S[k_{3 j},\, k_{3 j+1} -1] \cap 
   S[k_{3 j+1} + l^\alpha, k_{3 j+3}-1 ] \neq \emptyset \Big\}\\
  &\cup  \Big\{ S[k_{3j},\, k_{3 j+2}-l^\alpha] 
  \cap S[k_{3j+2} ,\, k_{3 j+3}-1] \neq \emptyset \Big\}.
\end{align*}
Then, 
\begin{align*}
a_2 &\leq \mathbb{P}\left( \bigcap_{j=1}^J (A_1^{j,\alpha} \cup (A_1^{j,\alpha})^c) \cap \{ N([k_{3j+1},k_{3j+2}-1], [k_{3j},k_{3j+3}-1])=0\} \right) \\
& \le \prod_{j=1}^J \mathbb{P} ( A_1^{j,\alpha} 
\cup 
((A_1^{j,\alpha})^c \cap \{ N([k_{3j+1},k_{3j+2}-1], [k_{3j},k_{3j+3}-1])=0\} )).
\end{align*}
By the same proof as that of \eqref{cc**}, 
\begin{align*}
\mathbb{P}(A_1^{j,\alpha}) \leq C (l^\alpha)^{-\frac{d}{2}+2}.
\end{align*}

Now, we can split ``$[k_{3j+1}, k_{3j+2}-1]$'' like in \eqref{cc*}.  
That is, let $L' = l^\alpha$, $J' = \frac{k_{3j+2} - k_{3j+1}}{2L'}$, 
$k'_u = u L'$ for $1\le u \le 2J'$. 
Then, we have 
\begin{align*}
&\mathbb{P}\Big( N([k_{3j+1},k_{3j+2}-1], [k_{3j},k_{3j+3}-1])=0, (A_1^{j,\alpha})^c \Big) \\
&\quad \leq  \bigg\{\sum_{u=1}^{2J-1} 
   \mathbb{P}\Big( S[k_{3j+1}, \, k_{3j+1}  + k'_{u} -1 ] 
   \cap S[k_{3j+1} + k'_{u +1}, \, k_{3j+2}-1] \neq \emptyset \Big) \bigg\} (:=\hat{a}_1)\\
&\qquad + \mathbb{P}\Big( \text{for all } 1 \leq u \leq 2J', \,
   S[k_{3j+1} ,\, k_{3j+1} +  k'_{u -1}] 
   \cap S[k_{3j+1} +k'_{u},\, k_{3j+2}-1] = \emptyset, \,\\
   & \quad \quad N([k_{3j+1}, k_{3j+2}-1],[k_{3j}, k_{3j+3}-1]) = 0,\\
   & \quad \quad S[k_{3j}, \, k_{3j+1}-1 ] \cap S[k_{3j+1}+k'_1, \, k_{3j+3}-1] = \emptyset, \\
   & \quad \quad S[k_{3j}, \, k_{3j+2}+k'_{2J'-1}-1 ] \cap S[k_{3j+2}, \, k_{3j+3}-1] = \emptyset \Big) 
   (:=\hat{a}_2).
\end{align*}
Note that
\begin{align*}
    \hat{a}_2
    \le & \mathbb{P}\Big( \text{for all } 0 \leq u \leq 2J'+1, \,
  S(k_{3j},\, k_{3j+1} +  k'_{u -1}] 
   \cap S[k_{3j+1} +k'_{u},\, k_{3j+3}) = \emptyset, \,\\
   & \quad \quad N([k_{3j+1}, k_{3j+2}-1],[k_{3j}, k_{3j+3}-1]) = 0 \Big).
\end{align*}
By the same proof as that of \cite[Lemma 1.1]{BSZ03}, we have
\begin{align*}
\hat{a}_1 \leq C \, l^{1-\alpha} \, (n^\alpha)^{-\frac{d}{2}+2}, 
\quad \hat{a}_2 \leq C \, l^{-\alpha}.
\end{align*}
Then, if we choose $\alpha$ close to $1$, for some $\delta > 0$ (note that the obtained value of $\delta$ does not depend on $J$), 
\begin{align*}
\hat{a}_1 + \hat{a}_2 \ll l^{-\delta} 
\end{align*}
and hence
\begin{align*}
a_2 \leq C(J) (l^{-\delta})^J.
\end{align*}

Therefore, if we pick sufficiently large $J$, 
we obtain
$a_2 \leq C l^{-\frac{d}{2}+2}$ and, in addition, the desired result,  equation \eqref{upper+}, follows by combining \eqref{cc*} and \eqref{cc**}. 
\end{proof}

\begin{proof}[Proof of \eqref{lower+}]

We start by defining a collection of sets which, when they hold together, will imply  $E_n > cn$ for an appropriate constant $c$. 
Let 
\begin{align*}
    &D_1:=\{\ X[n,2n] \ge c_dn/2  \},
    && D_2:=\{\ca(S[2n-sn,2n]) \ge c'_d sn  \},\\
    &D_3:=\{ \max_{2n-sn \le i \le 2n } \|S_i-S_n\| \le C_1 n^{1/2}   \},
    && D_4:=\{  \|S_{sn}-S_n\| \le C_1 n^{1/2}   \},
\end{align*}
{(see the definition of the capacity in \cite{DemboOkada} et al.).}
Here, $c_d$ and $c_d'$ are some dimension-dependent constants (by the law of large numbers for the capacity of the random walk range, we choose $c_d'$ to satisfy $\ca(S[0,n]) \sim 2c'_d n$). In addition, we choose  $0<s<c_d/8$. Recall the (slightly modified) definition of $T^1$ as the first cut point after time $n$. 
By \cite[Lemma 2.2]{Croydon2009b}, we know that there is $c_d>0$ such that $\mathbb{P}(X^2_{(-\infty,\infty)}[T^1,n+T^1]\ge c_dn)=1-o(1)$. 
Noting that $X^2_{(-\infty,\infty)}[a,b]$ will be a lower bound for $X[a,b]$ for any times $a$ and $b$ and that $X^2_{(-\infty,\infty)}[a,b]$ is monotonically increasing as we increase the interval $[a,b]$, we obtain 
\begin{align*}
\mathbb{P}(D_1) \ge &\mathbb{P}(T^1 \le n+ f(n),  X^2_{(-\infty,\infty)}[T^1,2n]\ge c_d n/2 )- \mathbb{P}(T^1 \ge n+ f(n))\\
\ge & \mathbb{P}( X^2_{(-\infty,\infty)}[0,n-f(n)]\ge c_d n/2 | 0 \text{ is cut point} )- \mathbb{P}(T^1 \ge n+ f(n))
= 1-o(1)
\end{align*}
if $f(n)=o(n)$ and $f(n)\to \infty$ as $n\to \infty$ with the aid of the estimate in the proof of \eqref{upper+}. 
Moreover, we have $\mathbb{P}(D_2)=1-o(1)$ (see \cite{JO68}) and for any $\epsilon>0$, there is $C_1<\infty$ such that $\mathbb{P}(D_i)\ge 1-\epsilon$ for $i=3,4$ by the local central limit theorem and the reflection principle for random walks (see \cite[Theorem 1.2.1]{Lawler1991}). 
Hence, if we pick sufficiently large $C_1<\infty$, we have 
$\mathbb{P}(\cap_{i=1}^4 D_i) \ge 1/2$. 
Let us define the set $A:=S[2n-sn,2n]$ and $z:=S_{sn}$.  

On the event $D_3 \cap D_4$, we see that we have $\max_{x\in A} \|x-z\| \le 2C_1n^{1/2}$. We now reverse the random walk $S_0,\ldots,S_{sn}$, to create a new random walk $\tilde{S}_0,\ldots,\tilde{S}_{sn}$, defined as,
\begin{equation*}
\tilde{S}_i := S_{sn-i}-S_{sn}.
\end{equation*}
$\tilde{S}$ is distributed as a random walk starting at the point $z$ and, conditional on the value of $z$, it is independent of $A$.

We now seek to compute the probability that the random walk $\tilde{S}$ intersects the set $A$. To do this, we derive a last hitting time decomposition of the random walk $\tilde{S}$ with the set $A$ before time $sn$, by considering the last time of intersection between $\tilde{S}$ and the set $A$. Then, 
\begin{align*}
    \mathbb{P}^z(T_{A}< sn )
    &= \sum_{x \in A} \sum_{i=1}^{sn} \mathbb{P}^z(\tilde{S}_i=x) \mathbb{P}^x(T_A >sn-i)\\
    \\& \ge \sum_{x \in A} \sum_{i=1}^{sn}  \mathbb{P}^z(\tilde{S}_i=x) \mathbb{P}^x(T_A = \infty) 
     \ge \min_{a \in A} \sum_{i=1}^{sn} \mathbb{P}^z(\tilde{S}_i=a) \sum_{x \in A} \mathbb{P}^x(T_A = \infty)\\
    &= \min_{a \in A} \sum_{i=1}^{sn} \mathbb{P}^z(\tilde{S}_i=a) \ca(A)
    \ge c(s) n^{-d/2+1} \ca(A). 
\end{align*}
In the last line, we used the local central limit theorem. It is important that $\|x-z\| < C n^{1/2}$, so for times $ \frac{sn}{2}\le i \le sn $, we know that $\min_{a \in A}\mathbb{P}^z(\tilde{S}_i =a) \ge \frac{c}{n^{d/2}}$, for some constant $c$ (cf. \cite[Theorem 1.2.1]{Lawler1991}). 

On the event $D_2$, we assumed that the capacity of $A$ is greater than $c_d'sn$. Thus, conditionally on the intersection of the events $D_2$, $D_3$, and $D_4$, we can derive a lower bound on the probability of intersection between $\tilde{S}$ (which is also $S[1, sn]$) and $A$ (which is $S[2n - sn, 2n]$) as
\begin{align*}
    \mathbb{P}(S[1,sn] \cap S[2n-sn,2n]\neq \emptyset 
    | \cap_{i=1}^4 D_i) \ge c(s) n^{-d/2+2}.
\end{align*}

We remark now that on this intersection event, we can derive the following lower bound on the cross term $E_n$. First, notice that,
\begin{equation*}
\begin{aligned}
& X[n,2n] \ge c_d  n/2, \quad
 X_n \ge 0,\quad 
 X_{2n} \le 2sn. 
\end{aligned}
\end{equation*}
The first line comes from the event $D_1$, and the second line is trivial.
To see the last line (for the graph distance), notice that 
if we let $S_{t_1}$ and $S_{t_2}$ be the points  of the intersection $\{S_1,\ldots,S_{sn}\}$ and $\{S_{2n-sn},\ldots,S_{2n}\}$ with $t_1 \in [0,sn]$ and $t_2 \in [2n -sn, 2n]$,  the shortest geodesic must be smaller than $t_1 + (2n-t_2)$. Namely, the shortest geodesic must be shorter than the path $S_0,\ldots, S_{t_1}, S_{t_2 +1},\ldots, S_{2n}$.  Now, $t_1$ and $2n-t_2$ are both less than $sn$. The same argument applies to the number of cut points. 


Thus, we have that $E_n =X_n+X[n+1,2n]-X_{2n} \ge (c_d/2- 2s)n  $. 
Therefore, 
\begin{align*}
\mathbb{P}(E_n\ge 2( c_d/2-s)n )\ge 
    \mathbb{P}(\{S[1,sn] \cap S[2n-sn,2n]\neq \emptyset\}
    \cap (\cap_{i=1}^4 A_i) ) \ge c(s) n^{-d/2+2}.
\end{align*}
\end{proof}

\begin{rem}\label{conv:ex1}
By \eqref{upper+}, we obtain $\mathbb{E}[X_n]/n \to c_0$ for some constant $c_0<\infty$ with the aid of Hammersley’s lemma for $d\ge 5$. 
Since $\mathbb{P}(S[1, \infty)\cap  \tilde{S}[0,\infty) \neq \emptyset)< 1$ for independent two random walks in $d\ge 5$ (e.g., see \cite[Theorem 1.2.1]{Lawler1991}), 
we have 
\begin{align*}
\mathbb{E}[T_{\mathcal{G}_n}]
    =\sum_{i=2}^n 
    \mathbb{P}(S[1,i-1 ]\cap  S[i,n] = \emptyset)
    \gtrsim n. 
\end{align*} 
Hence, we have $c_0>0$. 
\end{rem}


As a consequence, we obtain the following lemma.
\begin{lem}[$d=6$]\label{moment:d=6}
We have the following  estimates for the error terms $E_n$:
\begin{equation*}
\mathbb{E}[E_n] \lesssim(\log n), \quad 
\mathbb{E}[E_n^2] \asymp n.
\end{equation*}
As a consequence of the first and last inequality above, $\mathrm{Var}(E_n) \gtrsim n$. 

Let $\Omega_n$ be the event that $\mathbb{E}[E_n] - E_n \ge 0$, then, 
\begin{equation}\label{eq:bndneg}
\mathbb{E}[(E_n - \mathbb{E}[E_n])^2 \mathbbm{1}[\Omega_n]] \lesssim (\log n)^{2} \ll n.
\end{equation}
The previous  equation implies that the largest contribution to the variance of $E_n$ occurs when $E_n - \mathbb{E}[E_n]$ is positive. 

\end{lem}

\begin{proof}
Most of the previous estimates are a consequence of the layer cake representation of the expectation. For example, we have
\begin{equation*}
\mathbb{E}[E_n] = \sum_{l =1}^n \mathbb{P}(E_n \ge l) 
\lesssim \sum_{l=1}^n \frac{1}{l} \lesssim (\log n).
\end{equation*}
In addition, we obtain
\begin{equation*}
\mathbb{E}[E_n^2] = \sum_{l=1}^n l \mathbb{P}(E_n \ge l) 
\lesssim \sum_{l=1}^n  \frac{1}{l} l \lesssim  \sum_{l=1}^n 1 =  n.
\end{equation*}
Finally, we have,
\begin{equation*}
\mathbb{E}[E_n^2] = \sum_{l=1}^n l \mathbb{P}(E_n \ge l) \gtrsim \sum_{l= an}^{bn} l \frac{1}{l} \gtrsim n.
\end{equation*}
Since $E_n \ge 0$, we obtain that $\mathbb{E}[E_n] - E_n \lesssim \log n$. This gives us the final estimate we desire.
\end{proof}

\section{Proof of CLT for $d\ge 7$} 
In this section, we will show  the central limit theorem for $d\ge 7$. 
\subsection{Proof of variance}
In this subsection, we will give an estimate for the variance. 
\begin{lem}
    
For $d\ge 7$, 
\begin{align*}
    \mathrm{Var}(X^{(i)}_n) 
    \sim 
    \alpha_d^{i} n
\end{align*} 
with $\alpha_d^{i}>0$ for $i=1,2$. 
\end{lem}

To show it, we first prove the following lemma. 
\begin{lem}\label{variance}
For $d \geq 7$, the limit
\[
\lim_{n \to \infty} \frac{\operatorname{Var}( X_n )}{n}
\]
exists.
\end{lem}

For simplicity, we write $\overline{X}:=X-\mathbb{E}[X]$. 

\begin{proof}
We refer to \cite[Lemma 6.2]{LeGall1986}. 

Define
\[
a_k = \sup \left\{ \|\overline{X}_n\|_2 : 2^k \leq n < 2^{k+1} \right\}, 
\]
where $\|X\|_m :=\mathbb{E}[|X|^m]^{1/m}$ for $m\ge 1$. 
For $k \geq 2$, let $n$ be such that $2^k \leq n < 2^{k+1}$. 
Taking expectations and using the triangle inequality, we obtain
\[
\|\overline{X}_{2n}\|_2 \leq \|\overline{X}_n + \overline{X}[n,2n]\|_2 + 2\max\left( \|E_n\|_2, \mathbb{E}[E_n] \right). 
\]
Since $X_n$ and $X[n+1,2n]$ are independent, we have
\[
\|\overline{X}_n + \overline{X}[n+1,2n]\|_2 = \left( \|\overline{X}_n\|_2^2 + \|\overline{X}[n,2n]\|_2^2 \right)^{1/2}.
\]
Using \eqref{upper+} and that $\max\left( \|E_n\|_2, \mathbb{E}[E_n] \right) \leq C(\log n)^M$ with $M<\infty$ for $d \geq 8$ and $\max\left( \|E_n\|_2, \mathbb{E}[E_n] \right) \leq n^{1/4}$ for $d =7$, 
we find
\[
\|\overline{X}_{2n}\|_2 \leq \left( \|\overline{X}_n\|_2^2 + \|\overline{X}[n,2n]\|_2^2 \right)^{1/2} + c_1 n^{1/4}.
\]

From the definition of $a_k$, it follows that
\[
a_k \leq \sqrt{2} a_{k-1} + c_2 2^{k/4},
\]
for some constant $c_2$. Setting $b_k = \frac{a_k}{2^{k/2}}$, we get a recursive inequality $b_k \le b_{k-1} + c_1 2^{-k/4}$. This implies $a_k \lesssim 2^{k/2}$, and hence
\[
\operatorname{Var}(X_n) \lesssim n.
\]
Letting $\gamma_n = \mathbb{E}[(X_n - \mathbb{E}[X_n])^2]$, we have 
\[
\gamma_{n+m} \leq \gamma_n + \gamma_m + \delta_{n+m}, \quad \gamma_{n+m} \geq \gamma_n + \gamma_m - \delta_{n+m}', 
\]
where $\delta_{n+m}, \delta_{n+m}' = \mathcal{O}((n+m)^{3/4})$.
Indeed, if $n\le m$ and $E_{n,m}:= X_n+ X[n,n+m]-X_{n+m}$, 
\begin{align*}
\gamma_{n+m} = &\gamma_n + \gamma_m + 2 \mathbb{E}[(X_n - \mathbb{E}[X_n])(E_{n,m} - \mathbb{E}[E_{n,m}])] \\
+& \mathbb{E}[(X[n,n+m] - \mathbb{E}[X[n,n+m]])(E_{n,m} - \mathbb{E}[E_{n,m}])] + \mathbb{E}[(E_{n,m}- \mathbb{E}[E_{n,m}])^2].
\end{align*}
Then, the Cauchy-Schwarz inequality would give,
\begin{align*}
&\mathbb{E}[(X_n - \mathbb{E}[X_n])(E_{n,m} - \mathbb{E}[E_{n,m}])] \\
&\le \mathbb{E}[(X_n - \mathbb{E}[X_n])^2]^{1/2}\mathbb{E}[(E_{n,m} - \mathbb{E}[E_{n,m}])^2]^{1/2} \lesssim n^{1/2} m^{1/4}, \\
&\mathbb{E}[(X[n,n+m] - \mathbb{E}[X[n,n+m]])(E_{n,m} - \mathbb{E}[E_{n,m}])] 
\lesssim m^{1/2}  m^{1/4},\\
&\mathbb{E}[(E_{n,m} - \mathbb{E}[E_{n,m}])^2] \lesssim m^{1/2}
\end{align*}
and hence
\begin{align*}
    \max(\delta_{n+m}, \delta_{n+m}')
     \lesssim (m^{1/2}m^{1/4}+m^{1/2}).
\end{align*}
Therefore, by Hammersley’s lemma, the existence of the limit
\[
\lim_{n \to \infty} \frac{\gamma_n}{n}
\]
follows.
\end{proof}

To show that the above limit is positive, we will give a toy model to understand the proof. 

\begin{lem}[Toy model]
    We consider i.i.d. $(Y_i)_{i=1}^\infty$ such that $\mathbb{P}(Y_i=1)=p$ for $0<p<1$. 
    Set $T_i=\{s> T_{i-1}: Y_s \neq 1 \}$ 
    and $\hat{G}_m=\{l\le m: T_{l+1}-T_l = 2 \text{ or }3\}$, 
    $G_n=\{l: T_{l+1} \le n, T_{l+1}-T_l = 2 \text{ or }3\}$, 
    $\mathcal{\hat{F}}_m := \sigma ( (Y_{T_l})_{l=1}^{m} \cup  \cup_{l\notin \hat{G}_m} (Y_i)_{i=T_l+1}^{T_{l+1}-1} )$,   
    $\mathcal{F}_n := \sigma ( (Y_{T_l \wedge n})_{l=1}^{\infty} \cup  \cup_{l\notin G_n} (Y_i)_{i=T_l+1}^{T_{l+1}-1} )$. 
    Set $\hat{D}_m:=\{i\in [1, T_m]: Y_i=1 \}$, $D_n:=\{i\in [1, n]: Y_i=1 \}$. 
Let $\mathbb{P}_m$ and $\mathbb{E}_m$ be the probability and expectation for $(Y_{l})_{l=1}^{T_m}$ and 
    \begin{align*}
        \hat{N}_1 &:=|\{i \in [1, T_m] : i-1, i+1 \not\in \hat{D}_m, i \in \hat{D}_m \}|,\\ 
        \hat{N}_2 &:=|\{i \in [1, T_m] : i-1, i+2 \not\in \hat{D}_m, i,i+1 \in \hat{D}_m \}|. 
    \end{align*}
Moreover, let 
    \begin{align*}
        N_1 &:=|\{i \in [1, n] : i-1, i+1 \not\in D_n, i \in D_n \}|,\\ 
        N_2 &:=|\{i \in [1, n] : i-1, i+2 \not\in D_n, i,i+1 \in D_n \}|. 
    \end{align*}
    Namely, $N_1$ denotes the number of events in which $Y_i=1$ occurs exactly once, and $N_2$ denotes the number of events in which $Y_i=1$ appears exactly twice in a row. 
    $\mathbb{P}^n$ denotes the probability for $(Y_{T_l \wedge n})_{l=1}^{\infty}$.

    Now, we state our desired conclusions. For large $M$, there exists $c>0$ such that 
\begin{align}\notag
    &A_n:= \{ |N_1-\mathbb{E}[N_1]| \le M\sqrt{n}, \quad  
    |N_2-\mathbb{E}[N_2]| \le M\sqrt{n}\}, \\
    \label{comb:claim1}
    &\mathbb{P}(A_n) >c>0. 
\end{align}
In addition, for $|a-\mathbb{E}_m[\hat{N}_1]|\le \sqrt{m}$, $|b-\mathbb{E}_m[\hat{N}_2]|\le \sqrt{m}$ and $c\le \sqrt{m}$, 
\begin{align}\label{comb:claim2-}
    \mathbb{P}_m(\hat{N}_1 =a, \hat{N}_2 =b+ c |\mathcal{\hat{F}}_m) 
    \le C \mathbb{P}_m (\hat{N}_1 =a+c, \hat{N}_2 =b |\mathcal{\hat{F}}_m).
\end{align}
As a consequence, for $|a-\mathbb{E}[N_1]|\le M\sqrt{n}$, $|b-\mathbb{E}[N_2]|\le M\sqrt{n}$, 
\begin{align}\label{comb:claim2}
    \mathbb{P}^n(N_1 =a, N_2 =b+ \sqrt{n} ) 
    \le C \mathbb{P}^{n-\sqrt{n}} (N_1 =a+ \sqrt{n}, N_2 =b  ).
\end{align}
\end{lem}

\begin{proof}
It is easy to see that $\mathbb{P}(||G_n|-\mathbb{E}[|G_n|]| \le \sqrt{n})>c$. 
Since $(T_{l+1}-T_l)_{l=1}^m$ is i.i.d. and $\mathbb{E}_m[\hat{N}_1]=p(1-p)m$, $\mathbb{E}_m[\hat{N}_2]=p^2(1-p)m$, by the central limit theorem for the sum of Bernoulli random variables, there exists $c>0$ such that 
\begin{align*}
    &\hat{A}_m:= \{ |\hat{N}_1-\mathbb{E}_m[\hat{N}_1]| \le \sqrt{m}, \quad  
    |\hat{N}_2-\mathbb{E}_m[\hat{N}_2]| \le \sqrt{m}\}, \\
    &\mathbb{P}(\hat{A}_m) >c>0. 
\end{align*}

 We now define $m$ to be the largest value of $l$ such that $T_l \le n$. We consider values of $m$ that satisfy $|m- \mathbb{E}[|G_n|]| \le \sqrt{n}$.
Under these conditions, note that if $|\hat{N}_1-\mathbb{E}_m[\hat{N}_1]| \le \sqrt{m}$, then 
$|\mathbb{E}[N_1] - \mathbb{E}_m[\hat{N}_1] |+ \sqrt{m}= O(\sqrt{n})$ 
and hence $|N_1-\mathbb{E}[N_1]| = |\hat{N}_1-\mathbb{E}[N_1]| \le M\sqrt{n}$ for large $M<\infty$. 
In addition, if $|\hat{N}_2-\mathbb{E}_m[\hat{N}_2]| \le \sqrt{m}$, then $|N_2-\mathbb{E}[N_2]| \le M\sqrt{n}$ for  large $M<\infty$.

Then, 
\begin{align*}
    \mathbb{P}(A_n) 
    \ge &\mathbb{P}(||G_n|-\mathbb{E}[|G_n|]| \le \sqrt{n})\\
    \times &\mathbb{P}\left( |N_1-\mathbb{E}[N_1]| \le M\sqrt{n}, 
    |N_2-\mathbb{E}[N_2]| \le M\sqrt{n} \bigg|||G_n|-\mathbb{E}[|G_n|]| \le \sqrt{n}\right)\\
    \ge & c 
    \min_{  |m- \mathbb{E}[|G_n|]| \le \sqrt{n}} 
    \mathbb{P}\left(|N_1-\mathbb{E}[N_1]| \le M\sqrt{n}, 
    |N_2-\mathbb{E}[N_2]| \le M\sqrt{n} \bigg||G_n|  =m \right)\\
   \ge & c  \min_{  |m- \mathbb{E}[|G_n|]| \le \sqrt{n}} 
   \mathbb{P}_m(\hat{A}_m) \ge c
\end{align*}
 and hence we obtain \eqref{comb:claim1}. 
 
Noting that $\mathbb{P}_m(\cdot|\mathcal{\hat{F}}_m )$ is the conditional probability that specifies the value of the variable $Y_i$ except at those locations where $Y_i$ occurs exactly once or twice in a row. Under this conditional probability, we are potentially allowed to change these locations where $Y_i$ occurs exactly once to those  where $Y_i$ occurs twice in a row, or vice versa. 
Then, comparing the combinations, we notice that
\begin{align*}
   \frac{\mathbb{P}_m(\hat{N}_1 =a, \hat{N}_2 =b+ c |\mathcal{\hat{F}}_m)}
    {\mathbb{P}_m (\hat{N}_1 =a+c, \hat{N}_2 =b |\mathcal{\hat{F}}_m)}
    = \frac{(a+b+c)!}{a!(b+c)!}p^{a}p^{2(b+c)} 
    \times  \bigg(\frac{(a+b+c)!}{(a+c)!b!}p^{a+c}p^{2b} \bigg)^{-1}. 
\end{align*}
Hence, Stirling's formula yields \eqref{comb:claim2-}. 
Similarly, 
    \begin{align*}
   \frac{ \mathbb{P}^n(N_1 =a, N_2 =b+ \sqrt{n} )}
    {\mathbb{P}^{n-\sqrt{n}} (N_1 =a+\sqrt{n}, N_2 =b  )}
   = \frac{(a+b+\sqrt{n})!}{a!(b+\sqrt{n})!}p^{a}p^{2(b+\sqrt{n})} 
    \times  \bigg(\frac{(a+b+\sqrt{n})!}{(a+\sqrt{n})!b!}p^{a+\sqrt{n}}p^{2b} \bigg)^{-1}
\end{align*}
and Stirling's formula yields \eqref{comb:claim2}. 
\end{proof}

\begin{lem}
    It holds that $\alpha_d>0$ for $d\ge 7$. 
\end{lem}

\begin{proof}
Set $D:=\{i\in [0, n/2]: S_{2i}=S_{2i+2}=S_{2i+1}+e \}$, where $e$ is some unit vector, $Y_i:=1\{i\in D\}$ and $T_i=\{s> T_{i-1}: Y_s \neq 1 \}$. 
    Let 
    \begin{align*}
        N_1 &:=|\{i \in [1, n/2] : i-1, i+1 \not\in D, i \in D \}|,\\ 
        N_2 &:=|\{i \in [1, n/2] : i-1, i+2 \not\in D, i,i+1 \in D \}|,  
    \end{align*}
    and $\mathbb{P}^n$ denotes the probability for $(Y_{T_l \wedge (n/2)})_{l=1}^{\infty}$. 
$N_1$ counts the number of backtracks, while $N_2$ counts the number of double backtracks. We remark that changing a single backtrack to a double backtrack will not change the graph distance of the walk. However, the act of changing all single backtracks to double backtracks and vice versa would substantially change the length of the random walk. The existence of a procedure to change the volume without changing the graph distance essentially allows us to argue that the variance of the graph distance must be non-trivial. 
    
By \eqref{comb:claim1}, there exists $c_1>0$ such that 
\begin{align*}
    &A_n:= \{ |N_1-\mathbb{E}[N_1]| \le M\sqrt{n}, 
    \quad |N_2-\mathbb{E}[N_2]| \le M\sqrt{n} \}, \\
    &\mathbb{P}(A_n) >c_1>0. 
\end{align*}
In addition, by \eqref{comb:claim2}, for any $d>0$, 
\begin{align*}
& \mathbb{P}^n( X_n = d, A_n ) \\
= &  \sum_{\substack{|a-\mathbb{E}[N_1]| \le M\sqrt{n},\\
    |b-\mathbb{E}[N_2]| \le M\sqrt{n}}}
\mathbb{P}^n( X_n = d, N_1=a, N_2=b+\sqrt{n} ) \\
\le & c_2 \sum_{\substack{|a-\mathbb{E}[N_1]| \le M\sqrt{n},\\
    |b-\mathbb{E}[N_2]| \le M\sqrt{n}}}
\mathbb{P}^{n-\sqrt{n}}( X_{n-\sqrt{n}} = d, N_1=a+\sqrt{n}, N_2=b ) \\
  \le & c_2 \mathbb{P}(X_{n-\sqrt{n}} = d). 
\end{align*}

Now, assume for contradiction that for any $\epsilon>0$, as $n\to \infty$, 
\begin{align}\label{cont:conv}
    \mathbb{P}(|\overline{X}_n| \le \epsilon \sqrt{n})\to 1. 
\end{align}
Since $\mathbb{E}X_n/n \to c_0 $ by Remark \ref{conv:ex1}, we have $\limsup_{n\to \infty}(\mathbb{E}X_n- \mathbb{E}X_{n-\sqrt{n}})/\sqrt{n} \ge c_3 $ for some $c_3>0$. 
Indeed, if $\limsup_{n\to \infty}(\mathbb{E}X_n- \mathbb{E}X_{n-\sqrt{n}})/\sqrt{n} < c_3 $ and we set $a_m = \mathbb{E}X_{m^2}$,  
then $(a_m -a_{m-1})/m \le c_3$ for large $m$; hence, $a_m \le  c_0 m^2/2$ for large $m$ if we pick $c_3$ were chosen small enough. 
We now pick a subsequence $n$ such that $(\mathbb{E}X_n- \mathbb{E}X_{n-\sqrt{n}}) \ge c_3 \sqrt{n}$ and choose $\epsilon < c_3/2$. 
Then, 
\begin{align*}
    c_1/2 
    \le& \mathbb{P}(\{|\overline{X}_n| \le \epsilon \sqrt{n}\} \cap  A_n)\\ 
    \le & c_2 \mathbb{P}(X_{n-\sqrt{n}} \ge \mathbb{E}X_n - \epsilon \sqrt{n}  )\\
    \le & c_2  \mathbb{P}( X_{n-\sqrt{n}} \ge \mathbb{E}X_{n-\sqrt{n}} + c_3 \sqrt{n}/2  ). 
\end{align*}
Hence, we obtain 
\begin{align*}
     \mathbb{P}( X_{n-\sqrt{n}} \ge \mathbb{E}X_{n-\sqrt{n}} + c_3 \sqrt{n}/2  ) >c 
\end{align*}
and this contradicts \eqref{cont:conv}. 
Therefore, we obtain $\liminf_{n\to \infty}\mathbb{P}(|\overline{X}_n| \le \epsilon \sqrt{n})<1$ and hence $\alpha_d>0$. 
\end{proof}

\subsection{Proof of CLT}

\begin{lem}\label{2+e}
For $d \geq 7$ and $n \in \mathbb{N}$, we have
\[
\mathbb{E}[(\overline{X}_n)^{2+\epsilon}] \lesssim n^{1+\epsilon/2}.
\]
\end{lem}

\begin{proof}
Define
\[
a_k = \sup \left\{ \|\overline{X}_n \|_{2+\epsilon} : 2^k \leq n < 2^{k+1} \right\}.
\]
Fix $k \geq 2$, and let $n$ satisfy $2^k \leq n < 2^{k+1}$. Then, by the triangle inequality,
\[
\|\overline{X}_{2n}\|_{2+\epsilon}  
\le \|\overline{X}_{n} + \overline{X}[n,2n]\|_{2+\epsilon} +  \|E_n\|_{2+\epsilon}.
\]
Using the inequality
\begin{align*}
&\|\overline{X}_{n} + \overline{X}[n,2n]\|_{2+\epsilon} \\
\le &\big( \mathbb{E}[|\overline{X}_{n}|^{2+\epsilon}] + 
\mathbb{E}[|\overline{X}[n,2n]|^{2+\epsilon}] +  C\mathbb{E}[|\overline{X}_{n}|^{1+\epsilon}] \mathbb{E}[|\overline{X}[n,2n]|] \\&\hspace{1cm} 
 +C\mathbb{E}[|\overline{X}_{n}|] \mathbb{E}[|\overline{X}[n,2n]|^{1+\epsilon}]\big)^{1/(2+\epsilon)},
\end{align*}
and \eqref{upper+} and Lemma \ref{variance}, which gives $\mathbb{E}[|\overline{X}_{n}|^{1+\epsilon}] \mathbb{E}[|\overline{X}[n,2n]|] \lesssim n^{1+\epsilon/2}$, we obtain 
\[
\|\overline{X}_{2n}\|_{2+\epsilon} \le \left( \mathbb{E}[\overline{X}_{n}^{2+\epsilon}] + \mathbb{E}[\overline{X}[n,2n]^{2+\epsilon}] \right)^{1/(2+\epsilon)} + C n^{1/2}.
\]

From the definition of $a_k$, it follows that
\[
a_k \leq 2^{1/(2+\epsilon)} a_{k-1} + c_2 \cdot (2^{k})^{1/2},
\]
for some constant $c_2 > 0$. 
Define $b_k = 2^{-k/(2+\epsilon)}  a_k$. Then we get
\[
b_k \leq  b_{k-1} + c_2 2^{-k/(2+\epsilon)} (2^{k})^{1/2},
\]
which implies that $(b_k) \lesssim 2^{-k/(2+\epsilon)} (2^{k})^{1/2}$. Thus, $a_k \lesssim 2^{k/2} $, and hence for all $n$,
\[
\left( \mathbb{E}[\overline{X}_n^{2+\epsilon}] \right)^{1/(2+\epsilon)} \lesssim \sqrt{n} .
\]
This completes the proof.
\end{proof}

\begin{proof}[Proof of Theorem 1.1]
We decompose $X_n$ as follows: 
\[
X_n = \sum_{i=1}^{2^L} \tilde{X}_{(L,i)} + \mathcal{E},
\]
where $\tilde{X}_{(L,i)}:=X[(i-1)n2^{-L}, in2^{-L}]$ and  $\mathcal{E}:= -\sum_{k=0}^L \sum_{l=1}^{2^k} E_n^{(k,l)}$. 
Set $L = \lfloor \frac{1}{4} \log_2 n \rfloor$ so that $2^L = n^{1/4}$. Then,
\begin{align}\label{cross:assum1}
\frac{\mathbb{E}[|\mathcal{E}|]}{\sqrt{n}} \lesssim \frac{n^{1/4} \log n}{\sqrt{n}} \to 0 \quad \text{as } n \to \infty.
\end{align}
Hence, it suffices to prove that
\[
\frac{1}{\sqrt{n}} \sum_{i=1}^{2^L} \left( \tilde{X}_{(L,i)} - \mathbb{E}[\tilde{X}_{(L,i)}] \right) \xrightarrow{d} \mathcal{N}(0,\alpha_d).
\]

We now apply the Lindeberg–Feller central limit theorem. 
We write $Z_{L,i} := \tilde{X}_{(L,i)} - \mathbb{E}[\tilde{X}_{(L,i)}]$. 

We verify the Lindeberg–Feller conditions:

\textbf{(i) Variance Condition.} By Lemma 3.3, we have
\[
\sum_{i=1}^{2^L} \operatorname{Var}(\tilde{X}_{(L,i)}) \sim 2^L \cdot \alpha_d \cdot \frac{n}{2^L} = \alpha_d n,
\]
hence
\[
\frac{1}{n} \sum_{i=1}^{2^L} \operatorname{Var}(\tilde{X}_{(L,i)}) \to \alpha_d > 0.
\]

\textbf{(ii) Lindeberg Condition.} For any $\epsilon > 0$, we must show that as $n\to \infty$, 
\[
\sum_{i=1}^{2^L} \frac{1}{n} \mathbb{E}\left[ |Z_{L,i}|^2 \cdot \mathbf{1}_{\{|Z_{L,i}| > \epsilon \sqrt{n} \}} \right] \to 0.
\]
Using H\"{o}lder and Chebyshev’s inequality:
$$
\mathbb{E}\left[ |Z_{L,i}|^2 \cdot \mathbf{1}_{\{|Z_{L,i}| > \epsilon \sqrt{n} \}} \right] \le \mathbb{E}[|Z_{L,i}|^{2+\delta}]^{2/(2+\delta)} \mathbb{P}(|Z_{L,i}| \ge \epsilon \sqrt{n})^{\delta/(2+\delta)}.
$$
By Lemma \ref{2+e}, 
$$
\mathbb{E}[|Z_{L,i}|^{2+\delta}]^{2/(2+\delta)} \lesssim  \frac{n}{2^L}.
$$

Furthermore, Chebyshev's inequality should give,
$$
\mathbb{P}(|Z_{L,i}| > \epsilon \sqrt{n}) \le \frac{\mathrm{Var}(\tilde{X}_{L,i})}{\epsilon^2 n} \lesssim \frac{n}{\epsilon^2 n (2^L)} = \frac{1}{\epsilon^2 (2^L)}.
$$
Thus, we have,
$$
\mathbb{E}\left[ |Z_{L,i}|^2 \cdot \mathbf{1}_{\{|Z_{L,i}| > \epsilon \sqrt{n} \}} \right] \le \frac{n}{(\epsilon^2)^{\delta/(2+\delta)} (2^L)^{1 + \delta/(2+\delta)}}.
$$
Finally, we can sum up all of the appropriate terms and see that,
$$
\sum_{i=1}^{2^L} \frac{1}{n} \mathbb{E}\left[ |Z_{L,i}|^2 \cdot \mathbf{1}_{\{|Z_{L,i}| > \epsilon \sqrt{n} \}} \right] \lesssim 2^{L} \frac{1}{n} \frac{n}{(\epsilon^2)^{\delta/(2+\delta)}(2^L)^{1+ \frac{\delta}{(2+\delta)}}} \to 0,
$$
as $n \to \infty$ and $2^L \to \infty$.

Therefore, both Lindeberg–Feller conditions are satisfied, and we conclude
\[
\frac{\sum_{i=1}^{2^L} (\tilde{X}_{(L,i)} - \mathbb{E}[\tilde{X}_{(L,i)}])}{\sqrt{n}} \xrightarrow{d} \mathcal{N}(0,\alpha_d)
\]
as $n \to \infty$ and $2^L \to \infty$ and hence combining this with \eqref{cross:assum1}, we see that  
\begin{align*}
    \frac{X_n-\mathbb{E}X_n}{\sqrt{ \mathrm{Var}(X_n) }}
    \stackrel{d}{\rightarrow} \mathcal{N}(0,1)
\end{align*}
as $n \to \infty$. 
\end{proof}

\begin{rem}
We remark here that all the proofs in this section will hold without change as long as we assume the tail statistic $\mathbb{P}(X_n + X_{[n,2n]} - X_{2n} \ge l) \le C l^{-d/2 +2}$, for $d \ge 7$. This establishes Proposition \ref{prop1+} and Theorem \ref{m1+} for $d \ge 7$. 
\end{rem}

\section{6 Dimensions}

In this section, we estimate the variance and the limit in distribution of $\frac{X_n -\mathbb{E}[X_n]}{\sqrt{\mathrm{Var}(X_n)}}$ in $d=6$. 

\subsection{Lower and Upper Bounds on the Variance}

\subsubsection{Graph distance and cut points}

Now, we produce an upper bound to the variance of the graph distance.
\begin{lem} \label{lem:uprbndvar}
We have the following upper bound on the variance of $X_n$:   
\begin{equation*}
 \mathrm{Var}(X_n) \lesssim n (\log n). 
 \end{equation*}

\end{lem}
\begin{proof}
We now perform the full binary decomposition as,
\begin{equation*}
X_n = \sum_{k=0}^{\log_2 n} \sum_{l=1}^{2^k} -E_n^{(k,l)} + \sum_{k=0}^n X[k,k+1]. 
\end{equation*}
We define 
$$
\mathcal{E}_k := \sum_{l=1}^{2^k} -E_n^{(k,l)}, \quad 
\mathcal{E}:=\sum_{k=0}^{\log_2 n} \mathcal{E}_k, \quad 
MT:= \sum_{k=0}^n X[k,k+1].
$$
By independence of the components  $E_n^{(k,l)}$, we have that,
\begin{equation}\label{UB:var d=6}
\mathrm{Var}(\mathcal{E}_k) = \sum_{l=1}^{2^k} \mathrm{Var}(E_n^{(k,l)}) \lesssim \sum_{l=1}^{2^k}  \frac{n}{2^k}  = n .
\end{equation}
By expanding the variance and applying the Cauchy-Schwarz inequality, we obtain
\begin{equation}\label{variance:dec:d6}
\mathrm{Var}\left(\sum_{k=0}^{\log_2 n} \mathcal{E}_k\right) \le \sum_{k=0}^{\log_2 n} \mathrm{Var}(\mathcal{E}_k) + 2\sum_{k_1=0}^{\log_2 n}\sum_{k_2=0}^{\log_2 n} \mathrm{Var}(\mathcal{E}_{k_1})^{1/2} \mathrm{Var}(\mathcal{E}_{k_2})^{1/2} \lesssim   n (\log n)^2 +  n .
\end{equation}
In addition, since each $X[k,k+1]$ is deterministically 1, we have that $\mathrm{Var}(X[k,k+1]) =0$.
Then,
\begin{equation*}
 \mathrm{Var}(X_n) \lesssim n (\log n)^{2}. 
 \end{equation*}
By our binary decomposition, we find that, for $a_n=n/(\log n)^3$, 
\begin{equation*}
X_n = \sum_{l=1}^{a_n} X[(l-1)\frac{n}{a_n},l\frac{n}{a_n}] + \sum_{k=0}^{\log_2(a_n)} \sum_{l=1}^{2^k} (-E_n^{(k,l)}).
\end{equation*}

Furthermore, we remark that if the intervals $[ n/2^{k_1+1}(2l_1-2), n/2^{k_1+1} (2l_1)]$ (, which is the interval corresponding to the cross term $E^{(k_1,l_1)}$) and $[n/2^{k_2+1}(2l_2-2), n/2^{k_2 +1}(2l_2)]$(this is the interval corresponding to $E^{(k_2,l_2)}$) are completely disjoint, then $E^{(k_1,l_1)}$ and $E^{(k_2,l_2)}$ are independent and the covariance is $0$. 
Thus, the only non-zero covariances $\text{Cov}(E^{(k_1,l_1)},E^{(k_2,l_2)})$ are those with $k_2 \ge k_1$ and $ 2^{k_2 -k_1} l_1 \le l_2 \le 2^{k_2 - k_1} (l_1 +1)$.  

Now, we claim that 
\begin{align*}
    \mathbb{E}[E_n^{(k_1,l_1)} E_n^{(k_2,l_2)}]
    \lesssim (\log n)^2, 
\end{align*}
for most pairs of $(k_1,l_1)$ and $(k_2,l_2)$ with non-trivial intersection. 
Let $\tilde{S}$ be the reverse of the first part, that is $\tilde{S}_i^{(k,l)}= S_{l\frac{n}{2^k}-i} - S_{l\frac{n}{2^k}}$ . $\hat{S}$ will be the second part, that is $\hat{S}_i^{(k,l)} = S_{l\frac{n}{2^k}+i} - S_{l\frac{n}{2^k}}$. 
Indeed, if $E_n^{(k,l)} \ge a$, 
$\Omega (a,k,l)$ holds, where
$$
\begin{aligned}
\Omega (a,k,l):=& \{ \hat{S}^{(k,l)}[a/2, \frac{n}{2^k}] \cap \tilde{S}^{(k,l)}[0,\frac{n}{2^k}] \ne \emptyset \}  (=: A_1^{(a,k,l)})\\
& \bigcup \{ \tilde{S}^{(k,l)}[a/2,\frac{n}{2^k}] \cap \hat{S}^{(k,l)}[0,\frac{n}{2^k}] \ne \emptyset \}(=: A_2^{(a,k,l)}).
\end{aligned}
$$
Then, 
\begin{equation*} 
\mathbb{E}[E_n^{(k_1,l_1)} E_n^{(k_2,l_2)}] = \sum_{a,b} \mathbb{E}[ \mathbbm{1}[E_n^{(k_1,l_1)} \ge a] \mathbbm{1}[E_n^{(k_2,l_2)} \ge b] ]
\le \sum_{a,b}\sum_{i,j=1}^2 \mathbb{P}(A_i^{(a,k_1,l_1)} \cap A_j^{(b,k_2,l_2)}). 
\end{equation*}
In addition, 
\begin{align*}
&\mathbb{P}(A_1^{(a,k_1,l_1)} \cap A_1^{(b,k_2,l_2)})\\
 \le  &\mathbb{P}\bigg(\{\hat{S}^{(k_1,l_1)}\left[\frac{a}{2}, \frac{n}{2^{k_1}}\right] \cap \tilde{S}^{(k_1,l_1)}\left[0,\frac{n}{2^{k_1}}\right] \ne \emptyset\} \cap \{\hat{S}^{(k_2,l_2)}\left[\frac{b}{2}, \frac{n}{2^{k_2}}\right] \cap \tilde{S}^{(k_2,l_2)}\left[0,\frac{n}{2^{k_2}}\right] \ne \emptyset\} \bigg)\\
 \le & \sum_{u_1= \frac{a}{2}+ l_1\frac{n}{2^{k_1}}}^{(l_1+1)\frac{n}{2^{k_1}}} \sum_{u_2= (l_1-1)\frac{n}{2^{k_1}}}^{l_1\frac{n}{2^{k_1}}} \sum_{u_3= \frac{b}{2}+ l_2\frac{n}{2^{k_2}}}^{(l_2+1)\frac{n}{2^{k_2}}} \sum_{u_4= (l_2-1)\frac{n}{2^{k_2}}}^{l_2\frac{n}{2^{k_2}}} \sum_{x} \sum_{y}\\
 & \mathbb{P}(S_{u_1} = S_{u_2}=x, S_{u_3}=S_{u_4}=y). 
\end{align*}
By \cite[Theorem 1.2.1]{Lawler1991}, 
letting $p(x,n):=\frac{1}{(2d\pi n)^{d/2}}\exp(-\frac{|x|^2}{2dn})$ for $d=6$, 
we have that if $| \frac{2l_1 -1}{2^{k_1+1}} n - \frac{2l_2 -1}{2^{k_2+1}}n| > \frac{n}{2^{k_2+1}} $ for $k_1 < k_2$,
\begin{align*}
  &\sum_{a,b}\mathbb{P}(A_1^{(a,k_1,l_1)} \cap A_1^{(b,k_2,l_2)})\\
  \le & \sum_{a,b}\sum_{u_1= a/2+ l_1\frac{n}{2^{k_1}}}^{(l_1+1)\frac{n}{2^{k_1}}} \sum_{u_2= (l_1-1)\frac{n}{2^{k_1}}}^{l_1\frac{n}{2^{k_1}}} \sum_{u_3= b/2+ l_2\frac{n}{2^{k_2}}}^{(l_2+1)\frac{n}{2^{k_2}}} \sum_{u_4= (l_2-1)\frac{n}{2^{k_2}}}^{l_2\frac{n}{2^{k_2}}} \sum_{x} \sum_{y}\\
 & 1_{\{u_1 \le u_2 \le u_3 \le u_4\}} 
 Cp(x,u_1)p(0,u_2-u_1)p(y-x,u_3-u_2)p(0,u_4-u_3)
 \\
 +& 1_{\{u_1 \le u_3 \le u_2 \le u_4\}} 
 Cp(x,u_1)p(y-x,u_3-u_1)p(x-y,u_2-u_3)p(y-x,u_4-u_2)
 \ldots  \\
 \lesssim & (\log n)^2,
\end{align*}
where $\ldots$ represents other possible orderings of the variables $u_1,u_2,u_3,u_4$. 
In addition, 
if $| \frac{2l_1 -1}{2^{k_1+1}} n - \frac{2l_2 -1}{2^{k_2+1}}n| = \frac{n}{2^{k_2+1}} $, by Lemma \ref{moment:d=6}, 
\begin{align*}
  &\mathbb{E}[E_n^{(k_1,l_1)} E_n^{(k_2,l_2)}] 
 \le \mathbb{E}[(E_n^{(k_1,l_1)})^2]^{1/2} 
 \mathbb{E}[(E_n^{(k_2,l_2)})^2]^{1/2}
 \lesssim  \frac{n}{2^{k_2/2}2^{k_1/2}}. 
\end{align*}
We can obtain the same result for $\mathbb{P}(A_i^{(a,k_1,l_1)} \cap A_j^{(b,k_2,l_2)})$ by a similar computation for other cases. 
Hence, 
\begin{align*}
    &\sum_{k_1=0}^{\log_2(a_n)} \sum_{l_1=1}^{2^{k_1}} \sum_{k_2 =k_1+1}^{\log_2(a_n)} \sum_{l_2 =2^{k_2 -k_1}l_1}^{2^{k_2 - k_1}(l_1+1)} \mathbb{E}[E_n^{(k_1,l_1)} E_n^{(k_2,l_2)}]\\
    \le & \sum_{k_1=0}^{\log_2(a_n)} 2^{k_1}[a_n/2^{k_1}] (\log n)^2 \le  a_n(\log n)^3 \ll n(\log n).
\end{align*}
Therefore, with the aid of \eqref{UB:var d=6}, we obtain,
\begin{align}\label{disjont*}
&\mathrm{Var}\left(\sum_{k=0}^{a_n} \mathcal{E}_k\right) \\
\notag
\le & \sum_{k=0}^{\log_2 n} \mathrm{Var}(\mathcal{E}_k) + 2\sum_{k_1=0}^{\log_2(a_n)} \sum_{l_1=1}^{2^{k_1}} \sum_{k_2 =k_1+1}^{\log_2(a_n)} \sum_{l_2 =2^{k_2- k_1}l_1}^{2^{k_2 - k_1}(l_1+1)} 
\text{Cov}(E_n^{(k_1,l_1)}, E_n^{(k_2,l_2)})\\
\notag
\le & \sum_{k=0}^{\log_2 n} \mathrm{Var}(\mathcal{E}_k) + 2\sum_{k_1=0}^{\log_2(a_n)} \sum_{l_1=1}^{2^{k_1}} \sum_{k_2 =k_1+1}^{\log_2(a_n)} \sum_{l_2 =2^{k_2-k_1} l_1}^{2^{k_2 - k_1}(l_1+1)} \mathbb{E}[E_n^{(k_1,l_1)} E_n^{(k_2,l_2)}]
\lesssim   n (\log n)
\end{align}
and 
$$
\mathrm{Var}(MT) \le C a_n[n/a_n (\log (n/a_n))^2] \le Cn(\log \log n)^2.
$$
Finally, we can apply the Cauchy-Schwarz inequality to say that,
\begin{equation*} 
\mathrm{Var}(X_n) \le \mathrm{Var}(MT) + \mathrm{Var}(\mathcal{E})+2 \mathrm{Var}(MT)^{1/2} \mathrm{Var}(\mathcal{E})^{1/2} 
\lesssim n (\log n).
\end{equation*}
This completes the proof of the assertion.
\end{proof}

We start with the following lower bound.
\begin{lem} \label{lem:lwrbndvar}
There exists $0<c< \infty$ such that
\begin{equation*}
\mathrm{Var}(X_n) \ge c n \log n.
\end{equation*}
\end{lem}

\begin{proof}
Fix $\delta>0$ very close to 1. We will see where this condition  comes into play near the end of the proof. 
By our binary decomposition, we find that,
\begin{equation}\label{binary decomposition1}
X_n = \sum_{l=1}^{n^{\delta}} X[(l-1)n^{1-\delta},l n^{1-\delta}] + \sum_{k=0}^{\log_2(n^{\delta})} \sum_{l=1}^{2^k} (-E_n^{(k,l)}).
\end{equation}
We will write
$$
MT:= \sum_{l=1}^{n^{\delta}} X[(l-1)n^{1-\delta},l n^{1-\delta}] ,
\quad 
\mathcal{E}:=\sum_{k=0}^{\log_2(n^{\delta})} \sum_{l=1}^{2^k} (-E_n^{(k,l)}).
$$
We have that 
$$
\mathrm{Var}(X_n) = \mathrm{Var}(MT) +\mathrm{Var}(\mathcal{E}) + 2 \text{Cov}(MT, \mathcal{E})
$$
and 
\begin{equation}\label{eq:varcalE}
\mathrm{Var}(\mathcal{E}) =  \sum_{k=0}^{\log_2(n^{\delta})} \sum_{l=1}^{2^k} \mathrm{Var}(-E_n^{(k,l)}) + \sum_{k_1=0}^{\log_2(n^{\delta})} \sum_{l_1=1}^{2^{k_1}}\sum_{k_2=0}^{\log_2(n^{\delta})} \sum_{l_2=1}^{2^{k_2}}\text{Cov}(-E_n^{(k_1,l_1)}, -E_n^{(k_2,l_2)}).
\end{equation}
In addition, we obtain that, for sufficiently large $n$,
$$
\sum_{k=0}^{\log_2(n^{\delta})} \sum_{l=1}^{2^k} \mathrm{Var}(-E_n^{(k,l)}) \ge \sum_{k=0}^{\log_2(n^{\delta})} \sum_{l=1}^{2^k}  c \frac{n}{2^k} \ge c  \delta n \log n.
$$
By Lemma \ref{moment:d=6}, we have 
\begin{align*}
\text{Cov}(-E_n^{(k_1,l_1)},-E_n^{(k_2,l_2)}) 
=& -\mathbb{E}[E_n^{(k_1,l_1)}] \mathbb{E}[E_n^{(k_2,l_2)}]+
\mathbb{E}[E_n^{(k_1,l_1)}E_n^{(k_2,l_2)}]\\
\ge& - \mathbb{E}[E_n^{(k_1,l_1)}] \mathbb{E}[E_n^{(k_2,l_2)}] 
\ge - (\log n)^2.
\end{align*}

$$
\begin{aligned}
&\sum_{k_1=0}^{\log_2(n^{\delta})} \sum_{l_1=1}^{2^{k_1}} \sum_{k_2 =k_1+1}^{\log_2(n^{\delta})} \sum_{l_2 =2^{k_2 - k_1} l_1}^{2^{k_2 - k_1}(l_1+1)} -(\log n)^2 \ge  -\sum_{k_1=0}^{\log_2(n^{\delta})} 2^{k_1}[n^{\delta}/2^{k_1}] (\log n)^2 \\&
=- n^{\delta}(\log n)^3 \gg -n(\log n), \text{  if } \delta<1.
\end{aligned}
$$
By combining these estimates, we have, 
\begin{equation} \label{eq:lowerbndvarE}
\mathrm{Var}(\mathcal{E}) \ge c\delta n\log n.
\end{equation}
Furthermore, notice that by Lemma \ref{lem:uprbndvar}, we have the upper bound $\mathrm{Var}(X_n) \le C n \log n$. Thus, we can show that
$$
\mathrm{Var}(MT) \le C n^{\delta}[n^{1-\delta} \log (n^{1-\delta})] \le C(1-\delta) n \log n,
$$
 setting $\delta$ sufficiently close to $1$, this can be made very small.

Finally, we can apply the Cauchy-Schwarz inequality to say that,
\begin{equation*} 
\mathrm{Var}(X_n) \ge \mathrm{Var}(MT) + \mathrm{Var}(\mathcal{E})- 2 \mathrm{Var}(MT)^{1/2} \mathrm{Var}(\mathcal{E})^{1/2}.
\end{equation*}

Then for $\delta$ close to but not equal to $1$, we have
$\mathrm{Var}(\mathcal{E}) \ge c\delta n \log n$ and $\mathrm{Var}(MT) \le C(1-\delta) n \log n$. Thus, $\mathrm{Var}(X_n) \gtrsim  n \log n$, as desired.
\end{proof}


\subsubsection{General Analysis with heavy tailed Cross Terms }

In this section, we will just assume that we are considering a collection of variables $X_{[a,b]}$ that are measurable with respect to the walk $\{S_a,\ldots,S_b\}$ and that the cross term $E_n:= X_n + X_{[0,n]} - X_{2n}$ satisfies the tail bounds $\mathbb{P}(E_n \ge l) \asymp \frac{1}{l}$. 
We start with the following lower bound.
\begin{lem} \label{lem:lwrbndvar+}
There exist a constant $\epsilon >0$ and a subsequence $(n_k)$ such that $\mathrm{Var}(X_{n_k}) \ge \epsilon n_k \log n_k$.
In other words,
\begin{equation*}
\limsup_{n \to \infty} \frac{\mathrm{Var}(X_n) }{n \log n} > 0.
\end{equation*}
\end{lem}

\begin{proof}
Assume for contradiction that this is not true. Then, it must be the case that for every $\kappa>0$, there exists $N_{\kappa}$ such that for all $n \ge N_{\kappa}$, we have,
$$
\mathrm{Var}(X_n) \le \kappa n \log n.
$$
Fix $\delta>0$. For $\delta>0$, we will choose a sufficiently small $\kappa>0$ and derive a contradiction later.
Recall our binary decomposition in \eqref{binary decomposition1}. 

As long as $n^{1-\delta} \ge N_{\kappa}$,
we have that
\begin{equation*}
\mathrm{Var}(MT)  \le n^{\delta} [\kappa n^{1-\delta} \log n^{1-\delta}] = \kappa(1-\delta) n \log n.
\end{equation*}
In addition,  equation \eqref{eq:lowerbndvarE}, will give us $\text{Var}(\mathcal{E}) \ge c \delta n \log n.$
Finally, we can apply the Cauchy-Schwarz inequality to say that,
\begin{equation} \label{eq:Varlwrbnd}
\mathrm{Var}(X_n) \ge \mathrm{Var}(MT) + \mathrm{Var}(\mathcal{E})- 2 \mathrm{Var}(MT)^{1/2} \mathrm{Var}(\mathcal{E})^{1/2}.
\end{equation}
Notice that, as a function of the value of $\mathrm{Var}(\mathcal{E})$, the right-hand side increases when $\mathrm{Var}(\mathcal{E}) \ge \mathrm{Var}(MT)$.  
Indeed, notice that if $\mathrm{Var}(MT) < \frac{c \delta}{8} n \log n$, then the right hand side of \eqref{eq:Varlwrbnd} will be greater than,
\begin{equation*} 
(\mathrm{Var}(\mathcal{E})^{1/2} - \mathrm{Var}(MT)^{1/2})^2 \ge  \left( \frac{\sqrt{c  \delta}}{ \sqrt{2}} \sqrt{n \log n} - \frac{\sqrt{c  \delta}}{2\sqrt{2}} \sqrt{n \log n}\right)^2 \ge \frac{c  \delta}{8} n \log n.
\end{equation*}

Then, by setting $\kappa$ sufficiently small and applying the previous inequality, we have that $\mathrm{Var}(X_n) \ge \frac{c \delta}{8} n \log n$. This is a contradiction to the assumption that we cannot produce a lower bound to the $\limsup$ of $\frac{\mathrm{Var}(X_n)}{n \log n}$.
\end{proof}

Now, we produce an upper bound to the variance of general variables with the same heavy tail on the cross terms. 
Since the proof is the same as that of \eqref{variance:dec:d6}, we omit it. 
\begin{lem} \label{lem:uprbndvar+}
We have the following upper bound of the variance of $X_n$:   
\begin{equation*}
 \mathrm{Var}(X_n) \lesssim n (\log n)^{2}. 
 \end{equation*}

\end{lem}

\subsection{The Limit in Distribution of $\frac{X_n -\mathbb{E}[X_n]}{\sqrt{\mathrm{Var}(X_n)}} $}


\subsubsection{Graph distance and cut points}

\begin{proof}[Proof of Theorem \ref{m1} for $d=6$]

\textit{Part 1: A CLT contribution}

Define the value $n_{\delta}:= \exp\left[\frac{1}{2} \log n +  (\sqrt{\log n})^{1-\delta} \right]$. 
Now, 
\begin{equation*} 
\mathcal{G}:=\sum_{i=1}^{n_{\delta}}   X[(i-1) n/n_{\delta}, i n/n_{\delta}] - \mathbb{E}[  X[(i-1) n/{n_{\delta}}, i n/n_{\delta}]]
\end{equation*}
will be a normal random variable of variance of order $n\log n$. 
Define
$$
Z_{n_{\delta},i}:= \frac{  X[(i-1) n/n_{\delta}, i n/n_{\delta}] - \mathbb{E}[  X[(i-1) n/n_{\delta}, i n/n_{\delta}]]}{ \sqrt{\mathrm{Var}(X_{n/n_{\delta}}) }}.
$$
The variables $Z_{n_{\delta},i}$ are all independent of each other and one obtains
$$
\sum_{i=1}^{n_{\delta}}  \mathbb{E} \left[ Z_{n_{\delta},i}^2  \right] =1.
$$
Furthermore, we have that,
\begin{equation*}
\sum_{i=1}^{n_{\delta}} \mathbb{E}[Z_{n_{\delta},i}^2 \mathbbm{1}[ |Z_{n_{\delta},i}| \ge \epsilon ]] =0,
\end{equation*}
since deterministically $  X[(i-1) n/n_{n_{\delta}}, i n/n_{\delta}] \le n/n_{\delta} \ll \sqrt{n (\log n/n_{\delta}) }$ holds.

\textit{Part 2: Controlling the error term at mesoscopic scales}

Define $n_{\delta,-}:=\exp\left[ \frac{1}{2} \log n - (\sqrt{\log n})^{1- \delta} \right]$. 
Let us control
\begin{equation*}
 \mathcal{M}:=\sum_{k= \log_2 n_{\delta,-} }^{\log_2 n_{\delta}} \sum_{l=1}^{2^k}(-E_{n}^{(k,l)} + \mathbb{E}[E_{n}^{(k,l)}]).
\end{equation*} 
Since at a given level $k$, all of the $-E_n^{(k,l)}$ are independent of each other, we can compute
\begin{equation*}
\mathbb{E}\left[\left(\sum_{l=1}^{2^k}(-E_{n}^{(k,l)} + \mathbb{E}[E_{n}^{(k,l)}])\right)^2\right] = \sum_{l=1}^{2^k} \mathbb{E}[(-E_{n}^{(k,l)} + \mathbb{E}[E_{n}^{(k,l)}])^2]
\lesssim \frac{n}{2^k} 2^k = n. 
\end{equation*}

Thus, we see that $\mathbb{E}[\mathcal{M}^2] \le 4 n (\sqrt{ \log n})^{2 - 2 \delta}$. 
Hence, by Markov's inequality, we have that 
\begin{equation} \label{eq:Mesoscopic}
\mathbb{P}\bigg(\mathcal{M} \ge \sqrt{n (\log n)^{1-\delta/2}}\bigg) \to 0
\end{equation}
as $n \to \infty$ for any $\epsilon>0$.

\textit{Part 3: Controlling the error terms at the smallest scales}

We let $\Omega_{k,l}$ be the event that,
\begin{equation*}
 \Omega_{k,l}:= \left\{ |E_n^{(k,l)} - \mathbb{E}[E_n^{(k,l)}]|\ \le \left(\frac{n}{2^{k}} \right)^{1 - \frac{1}{(\sqrt{\log  n})^{1+ 2\delta}}}\right\}.
\end{equation*}
Notice that for fixed $k$, the events $\Omega_{k,l}$ are all independent of each other.
Define $\Omega_k:= \bigcap_{l=1}^{2^k} \Omega_{k,l}$
Thus, we see that when $2^k \le n_{\delta, -} \le n^{1/2}$, by \eqref{eq:bndneg}, 
\begin{equation*}
\begin{aligned}
&\mathbb{E}\left[ \left(\sum_{l=1}^{2^k} E_n^{(k,l)} - \mathbb{E}[E_n^{(k,l)}]\right)^2 \mathbbm{1}[\Omega_k] \right] = \sum_{l=1}^{2^k} \mathbb{E}[(E_n^{(k,l)} - \mathbb{E}[E_n^{(k,l)}])^2 \mathbbm{1}[\Omega_{k,l}]] \\&+ \sum_{l_1,l_2=1}^{2^k} \mathbb{E}[(E_n^{(k,l_1)} - \mathbb{E}[E_n^{(k,l_1)}]) \mathbbm{1}[\Omega_{k,l_1}]] \mathbb{E}[(E_n^{(k,l)} - \mathbb{E}[E_n^{(k,l)}]) \mathbbm{1}[\Omega_{k,l_2}]]\\
&\le 2^k \left(\frac{n}{2^k} \right)^{ 1- \frac{1}{(\sqrt{\log n})^{1+ 2\delta}}} + (2^k)^2 (\log n)^2 = n^{1 - \frac{1}{(\sqrt{\log n})^{1+2\delta}}}(2^{k})^{ \frac{1}{(\sqrt{\log n})^{1+2\delta}}} + n_{\delta,-}^2(\log n)^2 \\& \le  n^{1 - \frac{1}{2 (\sqrt{\log n})^{1+2\delta}}} +\exp[ \log n  - \sqrt{\log n}^{1-\delta} + 2\log \log n]. 
\end{aligned}
\end{equation*}

In addition, we also have that when $2^k \le n_{\delta,-}$,
\begin{equation*}
\begin{aligned}
&\mathbb{P}(\Omega_k^c) \le \sum_{l=1}^{2^k}
\mathbb{P}(\Omega_{k,l}^c) \le \frac{2^k}{\left(\frac{n}{2^{k}} \right)^{1 - \frac{1}{(\sqrt{\log  n})^{1+2\delta}}}} \\
=&\frac{ (2^{k})^{2 - \frac{1}{(\sqrt{\log n})^{1+2\delta}}}}{n^{1- \frac{1}{(\sqrt{\log n})^{1+2\delta}}}} \le  \frac{\exp\left[\left(\frac{\log n}{2} - \sqrt{\log n}^{1-\delta}\right) \left(2- \frac{1}{(\sqrt{\log n})^{1+2 \delta}}\right)\right]}{\exp [\log n -  (\sqrt{\log n})^{1- 2\delta}]} \le \exp [- (\sqrt{\log n})^{1-\delta}],
\end{aligned}
\end{equation*}
for $n$ sufficiently large.

If we now define $\Omega= \bigcap_{k=1}^{\log_2 n_{\delta,-}} \Omega_k$, then we see that,
\begin{equation*}
\mathbb{P}(\Omega^c) \le  \sum_{k=1}^{\log_2 n_{\delta,-}} \mathbb{P}(\Omega_k^c) \le \log n  \exp [- (\sqrt{\log n})^{1-\delta}] \ll 1.
\end{equation*}
Furthermore, if we define 
$$
\mathcal{L}:=\sum_{k=0}^{\log_2 n_{\delta,-}}\sum_{l=1}^{2^k} (E_n^{(k,l)} - \mathbb{E}[E_n^{(k,l)}]), 
$$
we have
\begin{equation*}
\begin{aligned}
&\mathbb{E}\left[ \mathcal{L}^2 \mathbbm{1}[\Omega] \right] \\
&\le \sum_{k_1,k_2=0}^{\log_2 n_{\delta,-}}\mathbb{E}\left[\left(\sum_{l=1}^{2^{k_1}} E_n^{(k_1,l)} - \mathbb{E}[E_n^{(k_1,l)}]\right) \left(\sum_{l=1}^{2^{k_2}} E_n^{(k_2,l)} - \mathbb{E}[E_n^{(k_2,l)}]\right) \mathbbm{1}[\Omega]\right]\\
& \le \sum_{k_1,k_2 =0}^{\log_2 n_{\delta,-}} \mathbb{E}\left[\left(\sum_{l=1}^{2^{k_1}} E_n^{(k_1,l)} - \mathbb{E}[E_n^{(k_1,l)}]\right)^2 \mathbbm{1}[\Omega_{k_1}] \right]^{1/2}\mathbb{E}\left[\left(\sum_{l=1}^{2^{k_2}} E_n^{(k_2,l)} - \mathbb{E}[E_n^{(k_2,l)}]\right)^2 \mathbbm{1}[\Omega_{k_2}] \right]^{1/2} \\&\le (\log n)^2 n^{1- \frac{1}{2 (\sqrt{\log n})^{1+2 \delta}}} + (\log n)^2\exp[ \log n  - \sqrt{\log n}^{1-\delta} + 4\log \log n]\ll n .
\end{aligned}
\end{equation*}
By using Markov's inequality, we see that,
\begin{equation} \label{eq:localerror}
\begin{aligned}
&\mathbb{P}\left(\sum_{k=0}^{\log_2 n_{\delta,-}}\sum_{l=1}^{2^k} (E_n^{(k,l)} - \mathbb{E}[E_n^{(k,l)}] >  \sqrt{n }\right)\\& \le \mathbb{P}(\Omega^c) + n^{-1}\mathbb{E}\left[ \left(\sum_{k=0}^{\log_2 n_{\delta,-}}\sum_{l=1}^{2^k} (E_n^{(k,l)} - \mathbb{E}[E_n^{(k,l)}]\right)^2 \mathbbm{1}[\Omega] \right] \to 0,
\end{aligned}
\end{equation}
as $n \to \infty$.

Recall that $\mathcal{G}$ from Part 1 was a Gaussian random variable, and the order of the variance of $\mathcal{G}$ was $n \log n$, which is much larger than the error terms of Parts 2 and 3. Then, combining Parts 1, 2, and 3, we can obtain the desired result using the Lindeberg–Feller central limit theorem.


\subsubsection{General Analysis with heavy tailed Cross Terms}
In this section, we will demonstrate a version of the phase transition demonstrated for the graph distance of the random walk range in six dimensions for variables whose cross terms have the same heavy tailed asymptotics: $\mathbb{P}(E_n \ge l) \asymp l^{-1}$. In particular, we do not make any assumptions about the covariance structure for the cross terms. We emphasize that the results of this section show that heavy tailed behavior is the key source of the phase transition.  We note that our results will apply to the effective resistance; we briefly give the definitions here.

We define a quadratic form $\mathcal{E}$ on functions $f,g : V(G) \to \mathbb{R}$ by
\begin{equation*}
  \mathcal{E}(f,g)
  = \frac12 \sum_{\substack{x,y \in V \\ \{x,y\} \in E}}
      (f(x)-f(y))(g(x)-g(y)).
\end{equation*}
This corresponds to the energy dissipation of an electrical network on the graph $G$. 
We also introduce the space
\begin{equation*}
  H^{2} := \{\, f \in \mathbb{R}^{V} : \mathcal{E}(f,f) < \infty \,\},
\end{equation*}
which serves as the natural domain of the energy form. 
The effective resistance between $A$ and $B$ is defined via the energy minimization 
\begin{equation*}
  R_{G}(A,B)^{-1}
  \;=\;
  \inf\Biggl\{
       \mathcal{E}(f,f)
     \;\Big|\; \quad 
     f|_A=1,\ f|_B=0, f\in H^2
  \Biggr\}. 
\end{equation*}
This quantity represents the voltage difference required to send the current
from $A$ to $B$, and therefore corresponds to the usual notion of effective
resistance in an electrical network.
For vertices $x,y\in V$, we write $R_{G}(x,y)$ to denote $R_{G}(\{x\},\{y\})$. 

Now, we will give the proof of Theorem \ref{m1+} for $d=6$. We present only the necessary adjustments from the proof of Theorem \ref{m1}. This involves changing Parts 1 and 4, while keeping Parts 2 and 3. As before, we mention that the particular value of this result is that it demonstrates that the right order of heavy tails is what ultimately leads to a phase transition in dimension 6, rather than detailed estimates of the covariance structure. 


We now define the function $f(n)$ as,
\begin{equation*}
\mathrm{Var}(X_n) =n (\log n)^{f(n)}.
\end{equation*}

\textit{Part 1: A CLT contribution}

Define the value $n_{\delta}:= \exp\left[\frac{1}{2} \log n +  (\sqrt{\log n})^{1-\delta} \right]$. 
Now, 
\begin{equation*} 
\mathcal{G}:=\sum_{i=1}^{n_{\delta}}   X[(i-1) n/n_{\delta}, i n/n_{\delta}] - \mathbb{E}[  X[(i-1) n/{n_{\delta}}, i n/n_{\delta}]]
\end{equation*}
will be a normal random variable of variance $n (\log (n/n_{\delta}))^{f(n/n_{\delta})}$, provided $f(n/n_{\delta}) \ge 0. $
Define
$$
Z_{n_{\delta},i}:= \frac{  X[(i-1) n/n_{\delta}, i n/n_{\delta}] - \mathbb{E}[  X[(i-1) n/n_{\delta}, i n/n_{\delta}]]}{ \sqrt{n (\log n/n_{\delta})^{f(n/n_{\delta})} }}.
$$
The variables $Z_{n_{\delta},i}$ are all independent of each other and one 
$$
\sum_{i=1}^{n_{\delta}}  \mathbb{E} \left[ Z_{n_{\delta},i}^2  \right] =1.
$$
Furthermore, we have that,
\begin{equation*}
\sum_{i=1}^{n_{\delta}} \mathbb{E}[Z_{n_{\delta},i}^2 \mathbbm{1}[ |Z_{n_{\delta},i}| \ge \epsilon ]] =0,
\end{equation*}
since deterministically, we have that $  X[(i-1) n/n_{n_{\delta}}, i n/n_{\delta}] \le n/n_{\delta} \ll \sqrt{n (\log n/n_{\delta})^{f(n/n_{\delta})} }$, provided $f(n^{1/2-\delta}) \ge 0$.

\textit{Part 4: Nontriviality of the Limit}

We will consider those values of $n$ such that $f(n/n_{\delta}) \ge 1- \delta/4$ (by Lemma \ref{lem:lwrbndvar}, this should happen infinitely often.) We will divide the analysis into two cases.

\textit{Case 1: $f(n/n_{\delta}) \ge 1- \delta/4$ and $f(n) \le f(n/n_{\delta}) +  \frac{1}{\log \log n}$ happen infinitely often together.}

Then, we have that 
$$
\begin{aligned}
&\mathrm{Var}(X_n) = n (\log n)^{f(n/n_{\delta})} (\log n)^{f(n) - f(n/n_{\delta})} \le n (\log (n/n_{\delta})^3)^{f(n/n_{\delta})} \log n^{\frac{1}{\log \log n}} \\&\le 27 e n (\log (n/n_{\delta}))^{f(n/n_{\delta})}.
\end{aligned}
$$
Since $\mathcal{G}$ is Gaussian with variance $n (\log (n/n_{\delta}))^{f(n/n_{\delta})}$,  with constant probability, it will be greater than $\sqrt{n (\log (n/n_{\delta}))^{f(n/n_{\delta})}}$. Equations \eqref{eq:localerror} and \eqref{eq:Mesoscopic} show that the remaining error terms will not be larger than $\epsilon\sqrt{n (\log n)^{1-\delta/2}}$ with high probability. Thus, on a set of non-vanishing probability, we must have that as $n\to \infty$, 
$$
\begin{aligned}
\mathbb{P} \bigg( \frac{\mathcal{M}  + \mathcal{L}}{  \sqrt{n (\log (n/n_{\delta}))^{f(n/n_{\delta})}} } \ge \epsilon \bigg) \to 0
\end{aligned}$$
for any $\epsilon>0$. 
Thus, $\frac{X_n- \mathbb{E}[X_n]}{\sqrt{\mathrm{Var}(X_n)}}$ converges to $\mathcal{N}(0,a)$. 

\textit{Case 2: $f(n/n_{\delta}) \ge 1- \delta/4$ and $f(n) \le f(n/n_{\delta}) +  \frac{1}{\log \log n}$ do not happen infinitely often together.} 

Then there exists a sequence $m_1,m_2,\ldots, m_k, \ldots$ such that $\frac{m_{k+1}}{(m_{k+1})_\delta} = m_k, \forall k$ and $f(m_{k+1}) - f(m_{k}) \ge \frac{1}{ \log \log m_{k+1}}$.  (Notice that by Lemma \ref{lem:lwrbndvar}, we will always be able to find elements such that $f(n/n_{\delta}) \ge 1- \delta/4$, so we do not need to worry about the possibility that $f(n/n_{\delta}) \le 1- \delta/4$ for all large $n$.)

The relationship $m_{k} = \frac{m_{k+1}}{(m_{k+1})_{\delta}}$ will imply that $m_{k+1} \le (m_k)^3$.  Thus, $m_{k+1} \le (m_1)^{3^k}$.
Therefore, we have that
$$
f(m_{k+1}) - f(m_k) \ge \frac{1}{\log \log m_{k+1}} 
\ge \frac{1}{(k+1) \log 4},
$$
for all $k$ large enough. 
This will imply that $\lim_{k \to \infty}f(m_{k}) = \infty$. This contradicts the upper bound of $f$ from Lemma \ref{lem:uprbndvar}.

Thus, we see that this proof establishes that there is a subsequence where
$
\frac{X_n - \mathbb{E}[X_n]}{\sqrt{\mathrm{Var}(X_n)}},
$
does not converge to $0$. Moreover, along this subsequence,  $X_n - \mathbb{E}[X_n]$ will be distributed (after appropriate normalization) as a normal random variable whose variance is given by $n (\log (n/n_{\delta}))^{f(n/n_{\delta})}$.
\end{proof}

\section{The fluctuations converge to 0 in dimension $d=5$}

In this section, we will give the proof of Propositions \ref{prop1} and \ref{prop1+} and Theorems \ref{m1} and \ref{m1+} in $d=5$. Since the proof is similar, we will omit the proof of Proposition \ref{prop1+} in $d=5$. 
\begin{lem} \label{lem:Var5d}
    For $d=5$, we have
    \begin{align*}
        \mathrm{Var}(X_n) \asymp  n^{3/2}. 
    \end{align*}
\end{lem}
\begin{proof}
Fix some small $\alpha>0$. 
We consider the variance of 
$\mathcal{E} = -\sum_{k=0}^{(1-\alpha)\log_2 n} \sum_{l=1}^{2^k} E_n^{(k,l)}$. 
For $0 \le k_1,k_2 \le (1-\alpha)\log_2 n$, 
$1\le l_1 \le 2^{k_1}$, $1\le l_2 \le 2^{k_2}$, let 
\begin{align*}
&\Lambda_1 := \{((k_1, l_1),(k_2, l_2)): 
[2^{-k_1}(l_1-1)n,2^{-k_1}l_1 n] \cap [2^{-k_2(l_2 -1)n},2^{-k_2}l_2n] 
\neq \emptyset  \},\\ 
&\Lambda_2 := \{((k_1, l_1),(k_2, l_2)): 
[2^{-k_1)}(l_1-1) n,2^{k_1}l_1n] \cap [2^{-k_2}(l_2-1)n,2^{-k_2}l_2 n] 
= \emptyset \}.   
\end{align*}
Since 
\begin{align*}
    \sum_{((k_1, l_1),(k_2, l_2)) \in \Lambda_2} \mathbb{E}[E_n^{(k_1,l_1)} E_n^{(k_2,l_2)} ] = \sum_{((k_1, l_1),(k_2, l_2)) \in \Lambda_2} \mathbb{E}[E_n^{(k_1,l_1)} ]\mathbb{E}[E_n^{(k_2,l_2)} ], 
\end{align*}
we have 
\begin{align*}
    \mathrm{Var} (\mathcal{E})
    \ge \sum_{k=0}^{(1-\alpha)\log_2 n} \sum_{l=1}^{2^k} \mathbb{E}[(E_n^{(k,l)})^2]
    - \sum_{((k_1, l_1),(k_2, l_2)) \in \Lambda_1} \mathbb{E}[E_n^{(k_1,l_1)} ]\mathbb{E}[E_n^{(k_2,l_2)} ]. 
\end{align*}
By (\ref{lower+}), we have $\sum_{k=0}^{(1-\alpha)\log_2 n} \sum_{l=1}^{2^k} \mathbb{E}[(E_n^{(k,l)})^2] \ge cn^{3/2}$. 
Now we claim that 
\begin{align*}
    \sum_{((k_1, l_1),(k_2, l_2)) \in \Lambda_1} 
    \mathbb{E}[E_n^{(k_1,l_1)} ]\mathbb{E}[E_n^{(k_2,l_2)} ] \le n^{3/2-\alpha/2+o(1)}. 
\end{align*}
Indeed, 
\begin{align*}
    &\sum_{((k_1, l_1),(k_2, l_2)) \in \Lambda_1} 
    \mathbb{E}[E_n^{(k_1,l_1)} ]\mathbb{E}[E_n^{(k_2,l_2)} ] \\
    \lesssim & (\log n)^2 \sup_{\alpha \le \beta_1 \le \beta_2 \le 1}
    \sum_{l_1,l_2: ((1-\beta_1) \log_2 n ,l_1)((1-\beta_2) \log_2 n, l_2)\in \Lambda_1}
    \mathbb{E}[E_n^{((1-\beta_1) \log_2 n,l_1)} ]\mathbb{E}[E_n^{{( 1- \beta_2)} \log_2 n,l_2)} ] \\
    \lesssim & (\log n)^2 \sup_{\alpha \le \beta_1 \le \beta_2 \le 1}
    n^{\beta_2-\beta_1}n^{1-\beta_2} n^{\beta_2/2}n^{\beta_1/2}
    \le n^{3/2-\alpha/2+o(1)}. 
\end{align*}
Hence, we obtain
$\mathrm{Var} (\mathcal{E})\gtrsim n^{3/2}$. 
In addition,  (\ref{lower+}) yields $\sum_{k=1}^{(1-\alpha)\log_2 n} \sum_{l=1}^{2^k} \mathbb{E}[(E_n^{(k,l)})^2] \le cn^{3/2}$. 
Now by the same argument as that of Lemma \ref{lem:uprbndvar}, we have 
\begin{align}\label{correlation:d=5}
    \sum_{k_1=0}^{(1-\alpha)\log_2 n} \sum_{l_1=1}^{2^{k_1}} \sum_{k_2 =k_1+1}^{(1-\alpha)\log_2 n} \sum_{l_2 =0}^{2^{k_2 - k_1}}
    \mathbb{E}[E_n^{(k_1,l_1)} E_n^{(k_2,l_2)} ] \le n^{3/2-\alpha/2+o(1)}. 
\end{align}
Hence, we obtain
$\mathrm{Var} (\mathcal{E})\lesssim n^{3/2}$. 

If we let $MT=\sum_{l=1}^{n^{1-\alpha}} X[(l-1)n^\alpha, ln^\alpha ]$, 
$\mathrm{Var} (MT)\le \sum_{l=1}^{n^{1-\alpha}} \mathbb{E}[X[(l-1)n^\alpha, ln^\alpha ]^2] \lesssim n^{1+\alpha}$ for sufficiently small $\alpha>0$. 
Then, $\mathrm{Var} (X_n) 
\ge (\mathrm{Var} (\mathcal{E})^{1/2}- \mathrm{Var} (MT)^{1/2})^2
\ge c n^{3/2}$ and $\mathrm{Var} (X_n) 
\le (\mathrm{Var} (\mathcal{E})^{1/2}+ \mathrm{Var} (MT)^{1/2})^2
\lesssim   n^{3/2}$. 
\end{proof}

\begin{proof}[Proof of Theorems \ref{m1} and \ref{m1+} in $d=5$]

For small $\delta>0$, we use Markov's inequality in the computation of the variance to derive,
\begin{align*}
    &\sup_{\alpha\le \beta \le 1-\delta}
    \mathbb{P}\bigg(\frac{\sum_{l=1}^{n^{1-\beta}}(E_n^{((1-\beta) \log n,l)}-\mathbb{E}[E_n^{((1-\beta) \log n,l)}])}{\sqrt{\mathrm{Var}(X_n) }} \ge \frac{\epsilon}{\log n} \bigg)\\
    \le& \sup_{\alpha\le \beta \le 1-\delta} (\log n)^2 \epsilon^{-2} \frac{\sum_{l=1}^{n^{1-\beta}}\mathrm{Var}(E_n^{((1-\beta) \log n,l)})}{\mathrm{Var}(X_n) } \\
\lesssim & \sup_{\alpha\le \beta \le 1-\delta}   (\log n)^2
 \frac{n^{1-\beta}n^{3\beta/2}}{n^{3/2}} 
\lesssim  (\log n)^2 n^{-\delta/2}.
\end{align*}
Now, we use Markov's inequality, along with the triangle inequality, to get,  
\begin{align*}
&\sup_{1-\delta \le \beta \le 1}
 \mathbb{P}\bigg(\frac{\sum_{l=1}^{n^{1-\beta}}(E_n^{((1-\beta) \log n,l)}-\mathbb{E}[E_n^{((1-\beta) \log n,l)}])}{\sqrt{\mathrm{Var}(X_n) }} \ge \frac{\epsilon}{\log n} \bigg)\\
    \le&  \sup_{1-\delta \le \beta \le 1}(\log n)\epsilon^{-1} \frac{\sum_{l=1}^{n^{1-\beta}} \mathbb{E}[E_n^{((1-\beta) \log n,l)}]}{\sqrt{\mathrm{Var}(X_n)} }\\
\lesssim & \sup_{1-\delta \le \beta \le 1} 
(\log n)
\frac{n^{1-\beta} n^{ \beta/2}}{n^{3/4}} 
\lesssim (\log n) n^{-1/4+\delta/2}. 
\end{align*}
Therefore, 
\begin{align*}
&\mathbb{P}\bigg( \frac{|X_n-\mathbb{E}[X_n]|}{\sqrt{\mathrm{Var}(X_n) }} \ge \epsilon \bigg)\\
    \le&  \mathbb{P}\bigg( \frac{|MT-\mathbb{E}[MT]|}{\sqrt{\mathrm{Var}(X_n) }} \ge \epsilon/2 \bigg)\\
   &+ C(\log n) \sup_{\alpha\le \beta \le 1}
    \mathbb{P}\bigg(\frac{\sum_{l=1}^{n^{1-\beta}}(E_n^{(\beta \log n,l)}-\mathbb{E}[E_n^{(\beta \log n,l)}])}{\sqrt{\mathrm{Var}(X_n) }} \ge \epsilon/(2\log n) \bigg)
   \lesssim n^{-\alpha/4}
\end{align*}
as desired. 
\end{proof}


\section{The fluctuations converge to 0 in dimension $d=4$} \label{sec:dim4}

In this section, we estimate the variance and the limit in distribution of $\frac{X_n -\mathbb{E}[X_n]}{\sqrt{\mathrm{Var}(X_n)}}$ in $d=4$. In subsection \ref{cross term}, we will give the lower bounds for the cross term, which holds for the graph distance and the cut points. In subsections \ref{upper:cross}-\ref{analysis of centered}, we will show the upper bounds for the cross term and main results in $d=4$. In subsection \ref{analysis of ER}, we will present our results for general random variables whose cross terms have the same heavy tailed behavior as the graph distance in dimension 4. 



\subsection{Lower bounds for the Cross Term}\label{cross term}

We will first start by obtaining a lower bound of $\mathbb{E}[E_n].$ Our heuristic argument is the following: if there is an intersection between $S[ 3n/2, 2n]$ and $S[0,n]$, then this will form a shortcut in the graph distance between the points $S_{2n}$ and $S_0$. The amount that the graph distance would reduce via this shortcut will be related to the graph distance between $S_{n}$ and $S_{2n}$, which, by a law of large numbers, should be $O\left(\frac{n}{(\log n)^{1/2}}\right)$. Furthermore, the probability of the intersection of $S[3n/2,2n]$ and $S[0,n]$ is approximately $O((\log n)^{-1})$. Thus, one would heuristically argue that $\mathbb{E}[E_n] = \frac{n}{(\log n)^{3/2\pm o(1)}}$.  

However, the exact value of the reduction of the graph distance is a complicated function of the loop intersection structure of the two walks. Thus, knowing that there is a long-range intersection may not be enough to deduce the reduction of the graph distance. Furthermore, to get the exact order $\frac{n}{(\log n)^{3/2}}$, that we desire, we would need to argue that there is an exact independence between the event that there is an intersection between $S[3n/2,2n]$ and $S[0,n]$ and the event that the graph distance between $S_n$ and $S_{2n}$ is $O\left( \frac{n}{(\log n)^{1/2}}\right) $. The following lemma will address these issues by finding appropriate events that will allow us to constrain the geometry of the random walk range.

\begin{lem} \label{lem:lowerbndexpdim4}
Let $E_n$ denote the cross term that appears when considering the number of cut points or the graph distance.
We have $\mathbb{E}[E_n] \ge \frac{n}{(\log n)^{3/2+o(1)}}$.

\end{lem}

\begin{rem}
    In \cite[Section 4]{Shiraishi2012}, he has already shown 
    $\mathbb{E}[E_n] \ge \frac{n}{(\log n)^{3/2+o(1)}}$ for the effective resistance. 
    We obtain the same result with a slightly simpler proof, which derives intersection estimates by using the capacity of the random walk. 
\end{rem}

\begin{proof}
We split the walk $S[0,2n]$ into two parts. $\tilde{S}$ will be the reversal of the first part, that is, $\tilde{S}_i= S_{n-i} - S_n$ . $\hat{S}$ will be the second part, that is $\hat{S}_i = S_{n+i} - S_n$.

We now define the following events that impose good properties on the first walk $\tilde{S}_n$. 
Pick small $\epsilon>0$. 
We let $\tilde{T}_i$ be the sequence of cut points along the walk $\tilde{S}_n$, and
\begin{equation*}
\begin{aligned}
& A_1:=\{ \max_{\tilde{T}_i \le n} (\tilde{T}_{i+1} - \tilde{T}_i) \le \frac{n}{(\log n)^{7/12}} \},
\quad A_2:=\{\tilde{X}_{n/2} \ge \frac{1}{1+\epsilon} \mathbb{E}[X_{n/2}] \},\\
&A_3:= \{\tilde{X}[n-n/(\log n)^{\epsilon}, n] \le n/(\log n)^{1/2+\epsilon/2} \}, \\
& A_4:= \{\ca(\tilde{S}(n-n/(\log n)^{\epsilon}, n)) \ge \frac{1}{2} \frac{\pi^2}{8}\frac{n}{(\log n)^{1+\epsilon}}  \},\\
&A_5:=\{ \tilde{S}_i \le C_1 \sqrt{n},\quad  \forall i \in [n-n/(\log n)^{\epsilon},n ]  \}.
\end{aligned}
\end{equation*}


In the above equation, we use the notation $\tilde{X}[a,b]$ to denote the graph distance between the points $\tilde{S}_a$ and $\tilde{S}_b$ along the graph $\tilde{S}[a,b]$.
By \cite[Lemma 2.3]{CroydonShiraishi2023},  $\mathbb{P}(A_1) = 1- o(1)$. Similarly, the law of large numbers ensures that $\mathbb{P}(A_2)$, $\mathbb{P}(A_3) = 1-o(1)$ (noting \cite[Theorem 1.2]{CroydonShiraishi2023}) and 
 $\mathbb{P}(A_4)=1-o(1)$ (see \cite[Theorem 1.1]{Chang}). By the local central limit theorem, we have that $\mathbb{P}(A_5) = 1- \epsilon(C_1)$, where $\epsilon(C_1) \to 0$ as $C_1 \to \infty$.  Let $\mathcal{Q}:= A_1 \cap A_2\cap A_3 \cap A_4 \cap A_5$. From  our previous estimates, we can take a union bound to deduce that $\mathbb{P}(\mathcal{Q}) \ge 1- 2 \epsilon(C_1)$, for $n$ sufficiently large.



Now, we need to define an event on the intersection of $\hat{S}$ with $\tilde{S}$. Let $\tau$ be defined as the first time of intersection between $\hat{S}$ and $\tilde{S}[n-n/(\log n)^{\epsilon},n]$ after time $n-n/(\log n)^{\epsilon}$. 
That is, $\tau:=\inf \{i> n-n/(\log n)^{\epsilon}: \hat{S}_i \in \tilde{S}[n-n/(\log n)^{\epsilon},n]  \}$. 
We define,
\begin{equation*}
\begin{aligned}
& B_1:= \{ \hat{X}[\tau, n] \le n/(\log n)^{1/2+\epsilon/2}\},
\quad B_2:= \{ \hat{S}_{n-n/(\log n)^{\epsilon} } \le C_1 \sqrt{n}\}, \\
& B_3:= \{ \hat{S}[n-n/(\log n)^{\epsilon}, n] \cap \tilde{S}[n-n/(\log n)^{\epsilon},n] \ne \emptyset\}.
\end{aligned}
\end{equation*}

As before, we let $\hat{X}[a,b]$ to denote the graph distance between the points $\hat{S}_a$ and $\hat{S}_b$ on the graph connecting $\hat{S}[a,b]$. Note that event $B_1$ is independent of $B_2 \cap B_3$ by the strong Markov property, even if you have conditioned on the walk $\tilde{S}[0,n]$.   By the law of large numbers, we have that $\mathbb{P}(B_1)= 1- o(1)$ and 
\begin{equation} \label{eq:prodprobability}
\begin{aligned}
 \mathbb{P}(B_1 \cap B_2 \cap B_3\cap \mathcal{Q})
 & = \mathbb{P}(B_1| B_2 \cap B_3 \cap \mathcal{Q}) \mathbb{P}(B_2 \cap B_3 \cap \mathcal{Q})\\&= \mathbb{P}(B_1) \mathbb{P}(B_3 | B_2 \cap \mathcal{Q}) \mathbb{P}(B_2 \cap \mathcal{Q}) .
\end{aligned}
\end{equation}

We first claim that we have $\mathbb{P}(B_3| B_2 \cap \mathcal{Q}) \ge \frac{c}{(\log n)^{1+2\epsilon}}$ for some constant $c>0$.
We use $\mathcal{S}$ as a shorthand for the set $\tilde{S}[n-n/(\log n)^{\epsilon},n]$.  
Let $\hat{S}_{n-n/(\log n)^{\epsilon}} \le C_1 n^{1/2}$, as is necessary to satisfy the event $B_2$.  By the last hitting time decomposition, we have that 
on $\mathcal{Q}$
\begin{equation*}
\begin{aligned}
&\mathbb{P}^{\hat{S}_{n-n/(\log n)^{\epsilon}}}(\tau < n/(\log n)^{\epsilon})\\
=&  \sum_{i=1}^{n/(\log n)^{\epsilon}} \sum_{a \in \mathcal{S}} \mathbb{P}^{\hat{S}_{n-n/(\log n)^{\epsilon}}}( \hat{S}_{i + n-n/(\log n)^{\epsilon}} = a) \mathbb{P}^{a}(\tau > n/(\log n)^{\epsilon}-i) \\
\ge& \sum_{i=1}^{n/(\log n)^{\epsilon}} \sum_{a \in \mathcal{S}} \mathbb{P}^{\hat{S}_{n-n/(\log n)^{\epsilon}}}( \hat{S}_{i + n-n/(\log n)^{\epsilon}} = a) \mathbb{P}^{a}(\tau = \infty)  \\
 \ge &\left(\min_{\hat{a} \in \mathcal{S}} \sum_{i = 1}^{n/(\log n)^{\epsilon}} \mathbb{P}^{\hat{S}_{n-n/(\log n)^{\epsilon}}}(\hat{S}_{i+n-n/(\log n)^{\epsilon}} = \hat{a})  \right) \sum_{a \in \mathcal{S}} \mathbb{P}^a(\tau = \infty)\\
= &\left(\min_{\hat{a} \in \mathcal{S}} \sum_{i = 1}^{n/(\log n)^{\epsilon}} \mathbb{P}^{\hat{S}_{n-n/(\log n)^{\epsilon}}}(\hat{S}_{i+n-n/(\log n)^{\epsilon}} = \hat{a})  \right) \ca(\mathcal{S}).
\end{aligned}
\end{equation*}

On the union $A_5 \cap B_2 $, we can apply the local central limit theorem to deduce that $\min_{\hat{a} \in \mathcal{S}} \sum_{i = 1}^{n/(\log n)^{\epsilon}} \mathbb{P}^{\hat{S}_{n-n/(\log n)^{\epsilon}}}(\hat{S}_{i+n-n/(\log n)^{\epsilon}} = \hat{a})  \ge c n^{-1}(\log n)^{-\epsilon}$, for some constant $c>0$ on $\mathcal{Q}$. Furthermore, the event $A_5$ ensures that $\ca(\mathcal{S}) \ge \frac{1}{2} \frac{\pi^2}{8}\frac{n}{(\log n)^{1+\epsilon}}$ on $\mathcal{Q}$. Multiplying these together, we see that there is some $c>0$ such that $\mathbb{P}^{\hat{S}_{n-n/(\log n)^{\epsilon}}}(\tau < n/(\log n)^{\epsilon} ) > c/(\log n)^{1+2\epsilon}$  on $\mathcal{Q}$. This means that $\mathbb{P}(B_3|B_2 \cap \mathcal{Q}) \ge c/(\log n)^{1+2\epsilon}$ for some $c>0$. Returning to equation \eqref{eq:prodprobability}, we see that we see that we must have $\mathbb{P}(B_1\cap B_2 \cap B_3 \cap \mathcal{Q}) \ge c/(\log n)^{1+2\epsilon}$ for some $c>0$.

Now, on the event $B_1 \cap B_2 \cap B_3\cap \mathcal{Q}$,  we claim we have $E_n > n/(\log n)^{1/2+\epsilon/2}$. 
Due to the presence of the cut points from the event $A_1$ and the lower bound on the graph distance $\tilde{X}_{n/2}$ from $A_2$, we have that $\tilde{X}_n \ge \frac{1}{1+\epsilon} \mathbb{E}[X_{n/2}] - \frac{n}{(\log n)^{7/12}}$. 
By events $A_3$, $A_1$, and $B_1$, the graph distance between the points $\hat{S}_n$ and $\tilde{S}_n$, is at most $2n/(\log n)^{1/2+\epsilon/2} +2\frac{n}{(\log n)^{7/12}}$.  Computing the difference, we now see that $E_n > n/(\log n)^{1/2+\epsilon/2}$ on $B_1 \cap B_2 \cap B_3 \cap \mathcal{Q}$ as desired and this event happens with probability $c/(\log n)^{1+2\epsilon}$. By multiplying these two terms together, we see that $\mathbb{E}[E_n]>n/(\log n)^{3/2+\epsilon/2+2\epsilon}$.  
This proof also works even for the number of cut points. 
\end{proof}

\begin{lem} \label{lem:setOmegahigprob}
There exists a set $\Omega$ with probability $\text{O}\left(\frac{\log \log n}{\log n} \right)$ such that we have 
\begin{equation*}
E_n \mathbbm{1}[\Omega^c] \ll \mathbb{E}[E_n].
\end{equation*}
\end{lem}

\begin{proof}
Recall the notation $\hat{S}$ and $\tilde{S}$ from the proof of Lemma \ref{lem:lowerbndexpdim4}. Consider the event
$$
\begin{aligned}
\Omega:=& \{ \hat{S}[n/(\log n)^4, n] \cap \tilde{S}[0,n] \ne \emptyset \} \bigcup \{ \tilde{S}[n/(\log n)^4,n] \cap \hat{S}[0,n] \ne \emptyset \}\\
& \bigcup \left\{ \sum_{i\in \left[\frac{n}{(\log n)^4}, \frac{n}{(\log n)^{2}}\right]} 
    1\{\tilde{S}[0,i-1 ]\cap  \tilde{S}[i, n] = \emptyset\}=0  \right\} \\
& \bigcup \left\{\sum_{i\in \left[\frac{n}{(\log n)^4}, \frac{n}{(\log n)^{2}}\right]} 
    1\{\hat{S}[0,i-1 ]\cap  \hat{S}[i, n] = \emptyset\}=0 \right\}.
\end{aligned}
$$

The first two events that compose $\Omega$ ensure that there exists some long-range intersection between $\tilde{S}$ and $\hat{S}$. The third and fourth  events force there to be no cut points in $\tilde{S}$ between $n/(\log n)^4$ and $n/(\log n)^2$.

We claim both that $\Omega$ holds with probability $ O\left( \frac{\log \log n}{\log n}\right)$  and that on the event $\Omega^c$, we must have that the cross term $E_n$ for the number of cut points and the graph distance must be less than $\frac{n}{(\log n)^2} \ll \mathbb{E}[E_n]$ by Lemma \ref{lem:lowerbndexpdim4}.

 By \cite[Theorem 4.3.6 part(ii)]{Lawler1991}, we must have that $\mathbb{P}( \hat{S}[n/(\log n)^4, n] \cap \tilde{S}[0,n] \ne \emptyset) = \text{O}\left(\frac{\log \log n}{\log n} \right)$.  Similarly, by \cite[equation (7)]{CroydonShiraishi2023}, 
$$\mathbb{P}\left(\sum_{i\in \left[\frac{n}{(\log n)^4}, \frac{n}{(\log n)^{2}}\right]} 
    1\{\tilde{S}[0,i-1 ]\cap  \tilde{S}[i, n] = \emptyset \}\right) = \text{O}\left( \frac{\log \log n}{\log n}\right). $$
    Combining these bounds gives us that $\mathbb{P}(\Omega) \le \text{O}\left( \frac{\log \log n}{\log n} \right)$, as desired.

Now, we will check that on $\Omega^c$, our cross term can be no more than $\frac{n}{(\log n)^2}$.
Notice that on $\Omega^c$, there must be some cut point in the interval $\left[\frac{n}{(\log n)^4}, \frac{n}{(\log n)^2} \right]$ for both $\tilde{S}$ and $\hat{S}$. Let $\mathcal{C}_{\tilde{S}}$ denote the time instance of the cut point of $\tilde{S}$ (cut time) in this interval, and let $\mathcal{C}_{\hat{S}}$ denote the cut time of $\hat{S}$ in this interval.  Furthermore, since there is no intersection between $\hat{S}[n/(\log n)^4,n]$ and $\tilde{S}$, $\mathcal{C}_{\hat{S}}$ must still be a cut time in the full graph $\hat{S} \cup \tilde{S}.$ 

Notice that we have $\hat{X}_n = \hat{X}[0,\mathcal{C}_{\hat{S}}] + \hat{X}[\mathcal{C}_{\hat{S}}, n] \le \hat{X}[\mathcal{C}_{\hat{S}}, n] + n/(\log n)^2.$
Similarly, we have $\tilde{X}_n \le \tilde{X}[\mathcal{C}_{\tilde{S}}, n]+n/(\log n)^2.$ Lastly,  $X_{2n} \ge \hat{X}[\mathcal{C}_{\hat{S}}, n]  + \tilde{X}[\mathcal{C}_{\tilde{S}}, n] $. Altogether, on $\Omega^c$, $E_n = \tilde{X}_n + \hat{X}_n - X_{2n} \le \frac{2n}{(\log n)^2} \ll \mathbb{E}[E_n]$. 
\end{proof}

\begin{rem}
If one is interested in the graph distance or the number of cut points, then we can proceed to more simply define $\Omega$ by considering only its first line from its definition in Lemma \ref{lem:setOmegahigprob}. We will use this fact later in the proof of Lemma  \ref{high prob.:intersection}.
\end{rem}

By applying the Cauchy-Schwarz inequality, we obtain the following lower bound on $\mathbb{E}[E_n^2]$ as a corollary.

\begin{cor}
Let $E_n$ denote the cross term that appears when computing the number of cut points or the graph distance. Then, we have,
\begin{equation*}
\mathbb{E}[E_n^2] \ge \frac{n^2}{(\log n)^{2+ o(1)}}.
\end{equation*}

\end{cor}

\begin{proof}
Let $\Omega$ be the set from Lemma \ref{lem:setOmegahigprob}.
Note that by the Cauchy-Schwarz inequality, 
\begin{equation*}
\mathbb{E}[E_n^2]\mathbb{P}(\Omega) \ge \mathbb{E}[E_n \mathbbm{1}[\Omega]]^2 \ge (1-o(1)) \mathbb{E}[E_n]^2.
\end{equation*}
Since $\mathbb{P}(\Omega) = O\left( \frac{\log \log n}{\log n}\right)$, we can take ratios and deduce that $\mathbb{E}[E_n^2] \ge \frac{n^2}{(\log n)^{2+o(1)}}$, as desired. 
\end{proof}

\subsection{Upper bounds for the Cross terms}\label{upper:cross}

Our proof of the upper bound for $\mathbb{E}[E_n^2]$ requires more steps. Indeed, unlike our proof of the lower bound, we cannot establish our upper bound directly. However, we will first show that if we have the inequality $\mathbb{E}[E_n^2] \le  \frac{C n^2\log \log n}{(\log n)^2} + (\log n)^{-\epsilon} \mathrm{Var}(X_n)$, this will be enough to show that $\mathrm{Var}(X_n) \le \frac{n^2}{(\log n)^{2-o(1)}}$ and ultimately that $\mathbb{E}[E_n^2] \le \frac{n^2}{(\log n)^{2-o(1)}}$.

\subsubsection{Our first reduction}

The goal of this section is to establish our reduction that showing $\mathbb{E}[E_n^2] \le \frac{C n^2\log \log n}{(\log n)^2}+ (\log n)^{-\epsilon} \mathrm{Var}(X_n)$ is sufficient to establish $\mathbb{E}[E_n^2] ,\mathrm{Var}(X_n) \le \frac{n^2}{(\log n)^{2+o(1)}}$. 

\begin{lem} \label{lem:uprboundEn}
Let $X_n$ denote the number of cut points or the graph distance. 
Assume that there is some $\epsilon >0$ such that for all $n$, we have 
\begin{equation} \label{eq:upren2dim4}
\mathbb{E}[E_n^2] \le \frac{C n^2\log \log n}{(\log n)^{2+o(1)}}  + (\log n)^{-\epsilon} \mathrm{Var}(X_n).
\end{equation}
Then, we must have that,
\begin{equation*}
\mathrm{Var}(X_n) \le \frac{n^2}{(\log n)^{2-o(1)}}, \quad 
\mathbb{E}[E_n^2] \le \frac{n^2}{(\log n)^{2-o(1)}}.
\end{equation*}
\end{lem}

\begin{proof}



Fix any $\kappa>0$;  we will first show that for $n$ large enough, we must have that,
\begin{equation*}
\mathrm{Var}(X_n) \le \frac{n^2}{(\log n)^{2-\kappa}}.
\end{equation*}
Let us write 
$$
\mathrm{Var}(X_{n}) = \frac{n^2}{(\log n)^{f(n)}}. 
$$
We will use the dyadic decomposition to find an equation for $f(n)$. We see that,
\begin{equation*}
\mathrm{Var}(X_{2n}) = \mathrm{Var}(X_n) + \mathrm{Var}(X_{[n,2n]}) + \mathrm{Var}(\mathcal{E}_n) +\text{Cov}(\mathcal{E}_n, X_{n}+X_{[n,2n]}).
\end{equation*}
By using the previous expansion of the variance and equation \eqref{eq:upren2dim4}, we obtain
\begin{equation*} 
\begin{aligned}
\frac{(2n)^2}{(\log (2n))^{f(2n)}} &\le \frac{2 n^2}{(\log n)^{f(n)}} +  \frac{C n^2 \log \log n}{(\log n)^{2+o(1)}} + (\log n)^{-\epsilon} \frac{n^2}{(\log n)^{f(n)}}\\& + \sqrt{\frac{ (\log n)^{-\epsilon} n^2}{(\log n)^{f(n)}} + \frac{C n^2 \log \log n}{(\log n)^{2+o(1)}}} \sqrt{\frac{2n^2}{(\log n)^{f(n)}}}.
\end{aligned}
\end{equation*}
We can bound the right-hand side as,
\begin{equation*}
\begin{aligned}
& \frac{2n^2}{(\log n)^{f(n)}} + \frac{ C n^2 \log \log n}{(\log n)^{2+o(1)}} + \frac{n^2}{(\log n)^{f(n)+\epsilon}} + C \frac{n^2}{(\log n)^{f(n)+\epsilon/2}} + C \frac{n^2 \sqrt{\log \log n}}{(\log n)^{f(n)/2 +1+o(1)}}\\
& \le  \frac{(2+\delta)n^2}{(\log n)^{f(n)}} + C \frac{n^2 \log \log n}{(\log n)^{2+o(1)}},
\end{aligned}
\end{equation*}
where $\delta>0$ can be set to be any small constant, and $C$ is some universal constant (that could depend on $\delta$) and could change from line to line.

We consider two cases:

\textbf{Case 1:} $f(n) \ge 2- \kappa$

In this case, note that for $n$ large enough, we must have that $\frac{n^2 \log \log n}{(\log n)^{2+o(1)}} \le \delta \frac{n^2}{(\log n)^{2-\kappa}}$. 
Thus, we obtain the inequality,
\begin{equation} \label{eq:case1fn}
\begin{aligned}
& \frac{4n^2}{(\log 2n)^{f(2n)}} \le \frac{(2+2 \delta)n^2}{(\log n)^{2 - \kappa}},\\
& \log \frac{4}{2+2 \delta}  + (2-\kappa) \log\log n \le f(2n)\log \log 2n,\\
&f(2n) \ge (2-\kappa) \frac{\log \log n}{\log \log n + \log \left[ 1+ \frac{\log 2}{\log n} \right]} + \frac{1}{\log \log 2n}\log \frac{4}{2+2 \delta}\\&\ge (2- \kappa) -  \frac{2-\kappa}{2\log n \log \log n} +  \frac{1}{\log \log  2n}\log \frac{4}{2+2 \delta} \ge (2-\kappa),
\end{aligned}
\end{equation}
when $n$ is large enough.

\textbf{Case 2: $f(n) \le 2 - \kappa$}

In this case, when $n$ is large enough, we can still obtain the following inequality for any small $\delta>0$: $\frac{n^2 \log \log n}{(\log n)^{2+o(1)}} \le \delta \frac{n^2}{(\log n)^{f(n)}}$. Thus, we obtain the inequality,
\begin{equation} \label{eq:case2fn}
\begin{aligned}
& \frac{4n^2}{(\log 2n)^{f(2n)}} \le \frac{(2+2\delta)n^2}{(\log n)^{f(n)}},\\
& (\log 2n)^{f(2n)} \ge \frac{4}{2+2\delta} (\log n)^{f(n)},\\
& f(2n) \log \log 2n \ge \log \frac{4}{2+2\delta} + f(n) \log \log n,\\
& f(2n) \ge f(n)\frac{\log \log n}{\log \log 2n} + \frac{1}{\log \log 2n}\log \frac{4}{2+2\delta} \\&\ge f(n) - \frac{f(n)}{2 \log n \log \log n} + \frac{1}{\log \log (2n)} \log \frac{4}{2+2\delta}  \ge f(n)  + \frac{c}{\log \log n},
\end{aligned}
\end{equation}
where $c>0$ will be some universal constant, and the least line will hold provided $n$ is sufficiently large.

Now, let $n_0$ be a large enough integer such that both equations \eqref{eq:case1fn} and \eqref{eq:case2fn} hold. Furthermore, consider $k$ large enough so that 
\begin{equation*}
\sum_{j=0}^{k-1} \frac{c}{\log \log (2^{j+1} n_0)} \ge 2- \kappa.
\end{equation*}

We claim that for any $n \ge 2^k n_0$, we must have that $f(n) \ge 2- \kappa$. First, consider the sequence $f(n), f(n/2), f(n/4), \ldots, f(n/2^L)$, where $L$ is the largest integer such that $n/2^L \ge n_0$. If there existed some $0 \le i \le L$ such that $f(n/2^i) \ge 2- \kappa$, then by repeatedly applying equation \eqref{eq:case1fn}, we must have that $f(n) \ge 2- \kappa$.

Otherwise, \eqref{eq:case2fn} must hold for all $0\le i \le L$, and we see that $f(n) \ge f(n/2^L) + \sum_{j=0}^{L} \frac{c}{\log \log (n/2^j)} \ge \sum_{j=0}^{k-1} \frac{c}{\log \log 2^{j+1}n_0} \ge 2-\kappa$. Here, we remark that we used the fact that $f(n) \ge 0$ is a trivial lower bound  for all $n$.

By substituting our bound on $\mathrm{Var}(X_n)$ into equation \eqref{eq:upren2dim4}, we derive an upper bound on $\mathbb{E}[E_n^2]$.
\end{proof}

Assuming the consequence of the above result, we also have the following corollary.
\begin{cor} \label{coruprboundendime4}
Assume that we know $\mathbb{E}[E_n^2] \le \frac{n^2}{(\log n)^{2+o(1)}}$, then we have $\mathbb{E}[E_n] \le \frac{n}{(\log n)^{3/2+o(1)}}$.
\end{cor}

\begin{proof}
Recall the set $\Omega$ from Lemma \ref{lem:setOmegahigprob}. By the Cauchy-Schwarz inequality and the definition of $\Omega$, we have
$$
\mathbb{E}[E_n^2] \mathbb{P}(\Omega) \ge \mathbb{E}[E_n \mathbbm{1}[\Omega]]^2 \ge (1-o(1))  \mathbb{E}[E_n]^2.
$$
Using the fact that $\mathbb{P}(\Omega) =O(\log \log n/ \log n)$ gives the result.
\end{proof}

\subsubsection{Establishing equation \eqref{eq:upren2dim4}}
The goal of this section is to establish equation \eqref{eq:upren2dim4}. We will first show that this equation holds provided that one derives an appropriate conditional statement for the probability that two random walks intersect.
Before we start, let us recall the notation that we used in the proof of Lemma \ref{lem:lowerbndexpdim4}.  Let us call the reversal of the walk $S[0,n]$ from time $0$ to $n$ $\hat{S}$. That is $\hat{S}_i = S_n - S_{n - i}$. Let us call $S[n, 2n]$ as $\tilde{S}$. That is $\tilde{S}_i = S_{n+i} - S_n$.

\begin{lem}\label{high prob.:intersection}
Let $X_n$ denote the number of cut points, or the graph distance.
Assume that there exists some $\epsilon>0$ such that, with probability  significantly less than $(\log n)^{-2}$, we have that,
\begin{equation} \label{eq:problnonintersect}
\mathbb{P}(\tilde{S}[n/(\log n)^4,n] \cap \hat{S} \ne \emptyset| \hat{S})  \ge (\log n)^{-\epsilon}.
\end{equation}
Then, equation \eqref{eq:upren2dim4} holds.
\end{lem}

\begin{proof}

\begin{enumerate}
\item Notice that $E_n> l$ implies that $\tilde{X}_n \ge \frac{l}{2}$ or $\hat{X}_n \ge \frac{l}{2}$. Furthermore, $\Omega$ also holds by Lemma \ref{lem:setOmegahigprob}. Then, 
$$
\{E_n > l\} \subset (\{\tilde{X}_n \ge l/2 \} \cap \Omega) \cup (\{\hat{X}_n \ge l/2 \} \cap \Omega).
$$

\item Furthermore, (again still considering that $E_n >l$), in the case $\tilde{X}_n  \ge \frac{l}{2}$,  there must be an intersection between $\tilde{S}{[n /(\log n)^4, n]}$ and $\hat{S}$ by the proof of Lemma \ref{lem:setOmegahigprob}.  (Recall, we needed the event $\Omega$ to hold in order to ensure that there is a large reduction in the graph distance.)

\item Vice versa, in the case $\hat{X}_n \ge \frac{l}{2}$, it must be the case that there is an intersection between $\hat{S}[n/(\log n)^4, n]$ and $\tilde{S}$.

\end{enumerate}

Thus, we have that,
\begin{equation*}
\mathbb{P}(E_n > l) \le 2\mathbb{P}(\tilde{X}_n \ge l/2, \tilde{S}{[n /(\log n)^4, n]} \cap \hat{S} \ne \emptyset).
\end{equation*}

Furthermore, observe,
\begin{equation} \label{eqLbounden2dim4}
\begin{aligned}
&\mathbb{E}[E_n^2] \le 2 \sum_{l=1}^{n} l \mathbb{P}(\tilde{X}_n \ge l/2, \tilde{S}[n/(\log n)^4, n] \cap \hat{S} \ne \emptyset ) \\
&\le 2 \sum_{l=1}^{4 \mathbb{E}[\tilde{X}_n]} 4 \mathbb{E}[\tilde{X}_n] \mathbb{P}(\tilde{S}[n/(\log n)^4, n] \cap \hat{S} \ne \emptyset) \\
&+ 2 \sum_{l> 4 \mathbb{E}[\tilde{X}_n]} l \mathbb{P}(\tilde{X}_n \ge l/2, \tilde{S}[n/(\log n)^4, n] \cap \hat{S} \ne \emptyset)\\
& \le \frac{C n^2 \log \log n}{(\log n)^{2+o(1)}} +  2 \sum_{l> 4 \mathbb{E}[\tilde{X}_n]} l \mathbb{P}(\tilde{X}_n \ge l/2, \tilde{S}[n/(\log n)^4, n] \cap \hat{S} \ne \emptyset) .
\end{aligned}
\end{equation}
We used $\mathbb{E}[\tilde{X}_n] \le \frac{n}{(\log n)^{1/2+o(1)}}$ to show the final inequality. (Note that in \cite[Lemma 2.2]{CroydonShiraishi2023}, they stated that this bound should be obtained for the graph distance if one follows the proof of \cite{Shiraishi2012}.) 
Now, assume that equation \eqref{eq:problnonintersect} holds with probability significantly less than $(\log n)^{-2}$. Let $\Omega_{B}$ be the set of walks $\tilde{S}$ such that equation \eqref{eq:problnonintersect} holds. We have that,
$$
\begin{aligned}
& \sum_{l> 4 \mathbb{E}[\tilde{X}_n]} l \mathbb{P}(\tilde{X}_n \ge l/2, \tilde{S}[n/(\log n)^4,n] \cap \hat{S} \ne \emptyset, \tilde{S} \not \in \Omega_B)\\
&+ \sum_{l> 4 \mathbb{E}[\tilde{X}_n]} l \mathbb{P}(\tilde{X}_n \ge l/2, \tilde{S}[n/(\log n)^4,n] \cap \hat{S} \ne \emptyset, \tilde{S}  \in \Omega_B)\\
&\le  n \sum_{l> 4 \mathbb{E}[\tilde{X}_n]} \mathbb{P}(\tilde{S} \not \in \Omega_B) \\
& + \sum_{l > 4 \mathbb{E}[\tilde{X}_n]} l \mathbb{P}(\hat{S} \cap \tilde{S}[n/(\log n)^4,n] \ne \emptyset| \tilde{S} \in \Omega_B, \tilde{X}_n > l/2) \mathbb{P}(\tilde{S} \in \Omega_B, \tilde{X}_n \ge l/2)\\
& \le o(n^2 (\log n)^{-2}) + \sum_{l > 4 \mathbb{E}[\tilde{X}_n]} l (\log n)^{-\epsilon} \mathbb{P}(\tilde{X}_n \ge l/2) \le  o(n^2(\log n)^{-2}) + 2 \mathrm{Var}(X_n) (\log n)^{-\epsilon}. 
\end{aligned}
$$
Combining the previous equation with equation \eqref{eqLbounden2dim4} gives the desired result.
\end{proof}

\subsection{Establishing equation \eqref{eq:problnonintersect} }

\begin{lem} \label{lem:pfeqnonintersect}

With probability greater than $1- \frac{1}{(\log n)^2}$ on walks $\tilde{S}$, there exists some $c>0$ such that we have,
\begin{equation*}
\mathbb{P}(\hat{S} \cap \tilde{S}[\frac{n}{(\log n)^4}, n] \neq \emptyset| \tilde{S}) \le (\log n)^{-c}.
\end{equation*}

\end{lem}

We start by defining a sequence of scales $0=K_0 < K_1< K_2 < K_3< \ldots< K_f$; also, let $A(R_o,R_i)$ denote the annulus around the origin with outer radius $R_o$ and inner radius $R_i$. 
Let $\mathcal{A}_{i,n}:=A[\sqrt{n/K_i}, \sqrt{n/K_{i+1}}]$. 

By considering the last time of intersection between $\hat{S}$ and the set $\tilde{S}[\frac{n}{(\log n)^4},n]$, we can derive the following bound on the probability of intersection.
\begin{align} \notag
&\mathbb{P}(\hat{S} \cap \tilde{S}[\frac{n}{(\log n)^4}, n] \neq \emptyset| \tilde{S})\\
\label{eq:capacitydiv-}
\le &\sum_{i=0}^{10} \text{Cap}(\tilde{S}[\frac{n}{(\log n)^4},n] \cap \mathcal{A}_{i,n}) \max_{z \in \mathcal{A}_{i,n}} G(z)\\
\label{eq:capacitydiv}
+ &\sum_{i=11}^{f-1} \text{Cap}(\tilde{S}[\frac{n}{(\log n)^4},n] \cap \mathcal{A}_{i,n}) \max_{z \in \mathcal{A}_{i,n}} G(z). 
\end{align}

Our heuristics for analyzing \eqref{eq:capacitydiv} are the following:
\begin{enumerate}
\item The capacity of $\text{Cap}(\tilde{S}[\frac{n}{(\log n)^4},n] \cap \mathcal{A}_{i,n})$ can be bounded by $ |\tilde{S}[\frac{n}{(\log n)^4},n] \cap \mathcal{A}_{i,n}|$. Namely, the capacity will be bounded by the volume.

\item The volume can be bounded very closely by  $\frac{n}{K_{i}^{1+\delta}}$, which is related to the expected local time of the random walk $\tilde{S}$ in the annulus $ \mathcal{A}_{i,n}$.

\item For \eqref{eq:capacitydiv-}, we will use improved estimates for the capacity of random walks.

\end{enumerate}

Eventually, we will choose $K_i= \frac{1}{(\log n)^{i(1-\epsilon)}}$ for $i \le 11$, and after this, we will choose $K_i =2 K_{i-1}$ for $i \ge 12$. 

Furthermore, we can choose $K_f = n^{1/10}$. Note that by the local central limit theorem, the probability that a point $\tilde{S}_k$ for $k \in [n/(\log n)^4, n]$ lies in the ball of radius $n^{1/10}$ around the origin is at most $\frac{n^{2/10}}{\frac{n^2}{(\log n)^8}}$. Taking a union bound, the probability that there is a point in $\tilde{S}[\frac{n}{(\log n)^4}, n]$ that lies in the ball of radius $n^{1/10}$ around the origin is at most $\frac{n n^{1/5}(\log n)^8}{n^2} \ll (\log n)^{-2}$.

\subsubsection{The large scales}

We will first begin our analysis of equation \eqref{eq:capacitydiv} by considering the largest scales after $K_{11}$. 

We  need the following lemma. 
\begin{lem} \label{lem:localtimebnd}
Let $S[0,n]$ be a random walk starting at the origin in $d=4$. Let $\mathcal{B}_r$ be a ball of radius $r$ centered at the origin. Let $L(\mathcal{B}_r)$ be defined as,
\begin{equation*}
L(\mathcal{B}_r):= \sum_{t=\frac{n}{(\log n)^4}}^n \mathbbm{1}[S_t \in \mathcal{B}_r].
\end{equation*}
Namely, it is the number of times that $S_t$ lies in the ball $\mathcal{B}_r$.

Then, we have that 
\begin{equation} \label{eq:ballexpect}
\mathbb{E}[L(\mathcal{B}_r)] \le C \frac{r^4 (\log n)^4}{n}.
\end{equation}

Furthermore, we would also have,
\begin{equation} \label{eq:momball}
\mathbb{E}[L(\mathcal{B}_r)^k] \le C^k \frac{r^4 (r^2)^{k-1}(\log n )^4}{n}.
\end{equation}
\end{lem}
\begin{proof}
By using the local central limit theorem, the probability that a point, $S_t$, lies in the ball $\mathcal{B}_r$ is of order $\frac{r^4}{t^2}$. If we sum up this probability over all times $t$ from $n/(\log n)^4$ and $n$, we see that the expectation of the local time within this ball is $ \frac{r^4(\log n)^4}{n}$.

Now, if we want to consider the expectation of the $k$th power of the local time. We need to consider the probability that there are k distinct times $t_1<t_2<\ldots<t_k$ such that $S_{t_i}$ are all in the ball with radius $r$. First, we remark that given that $S_{t_1}$ is in $\mathcal{B}_r$, the probability that $S_{t_2} \in \mathcal{B}_r$ is at most $\text{O}\left(\frac{r^4}{(t_2-t_1)^2}\right)$ if $t_2 -t_1 \ge r^2$ or $1$ if $t_2 -t_1 \le r^2$. Thus, we see a bound on $\mathbb{E}[L(\mathcal{B}_r)^k]$ would be the following.
$$
\sum_{t_1= n/(\log n)^4}^n \sum_{t_2 -t_1 =1}^n \ldots \sum_{t_k - t_{k-1}=1}^n \frac{C r^4 (\log n)^4}{n} \prod_{i=1}^{k-1} \frac{C r^4}{\max(r^4, (t_{i+1}-t_i)^2)} \le C^k \frac{r^4 (r^2)^{k-1}(\log n)^4}{n}.
$$

Here, we used the fact that $\sum_{t=1}^n \frac{C r^4}{\max(r^4,t^2)} \le C r^2$. 
\end{proof}

\begin{proof}[Proof of Lemma \ref{lem:pfeqnonintersect}: Estimate for \eqref{eq:capacitydiv}]

Define the event
\begin{equation*}
\mathcal{O}_i:= \{L(\mathcal{B}_{\sqrt{n/K_i}}) \le \frac{n}{(K_i)^{1+\delta}} \},
\end{equation*}
for some small $\delta>0$. By using Markov's inequality and equation \eqref{eq:ballexpect}, we see that we have
\begin{equation*}
 \mathbb{P}(\mathcal{O}_i^c) \le \frac{(\log n)^4}{K_i^{1-\delta}}.
\end{equation*}
Thus, we see that,
\begin{equation*}
\mathbb{P}((\bigcap_{k=11}^f \mathcal{O}_k)^c) \le \sum_{k=11}^f \frac{(\log n)^4}{K_k^{1-\delta}} \le C \frac{(\log n)^4}{K_{11}^{1-\delta}}\ll \frac{1}{(\log n)^2}.
\end{equation*}
In  what follows, we will restrict our study to the event 
\begin{equation} \label{eq:deftileomega}
\tilde{\Omega}:= \bigcap_{k=11}^f \mathcal{O}_k.
\end{equation}
On this event, we will analyze \eqref{eq:capacitydiv}. 
Notice that we must have $|\tilde{S}[\frac{n}{(\log n)^4},n] \cap \mathcal{A}_{i,n})| \le |L(\mathcal{B}_{\sqrt{n/K_i}})| \le \frac{n}{(K_i)^{1+\delta}} \ll 1$.  By  using subadditivity of the capacity and the fact that the capacity must be less than the volume, we must have that
$$
\begin{aligned}
\text{Cap}(\tilde{S}[\frac{n}{(\log n)^4},n] \cap \mathcal{A}_{i,n}) \le |\tilde{S}[\frac{n}{(\log n)^4},n] \cap \mathcal{A}_{i,n}| 
\le |L(\mathcal{B}_{\sqrt{n/K_i}})| \le  \frac{n}{(K_i)^{1+\delta}}.
\end{aligned}
$$

Now, since $ \max_{z \in \mathcal{A}_{i,n}} G(z) \le C\frac{K_{i+1}}{n}$, then  we have,
\begin{equation*}
\begin{aligned}
\sum_{i=11}^{f-1} \text{Cap}(\tilde{S}[\frac{n}{(\log n)^4},n] \cap \mathcal{A}_{i,n}) \max_{z \in \mathcal{A}_{i,n}} G(z) \le \sum_{i=11}^{f-1}  \frac{n}{(K_i)^{1+\delta}} \frac{K_{i+1}}{n} \le C K_{11}^{-\delta}.
\end{aligned}
\end{equation*}
\end{proof}

\subsubsection{The smallest scales}

We now continue our analysis of equation \eqref{eq:capacitydiv} by considering the contributions from the large scales of $K_i$ with $i \le 11$.  We start by giving some lemmas that will be useful for our analysis. In the following, it is assumed that we consider an index with $i \le 11$.

Notice that by the moment estimate for the capacity of the random walk in four dimensions, we have the following lemma.
\begin{lem}
Let $S[0,n]$ be a random walk starting at the origin. For any integer $k>0$ and $\chi>0$, we have that,
\begin{equation} \label{eq:weakercapbound}
\mathbb{P}\left(\ca(S[0,n]) \ge \frac{n}{(\log n)^{1-\chi}}\right) \le \frac{1}{(\log n)^k},
\end{equation}
for any $k \ge 0$.
\end{lem}

\begin{proof}

In \cite[Lemma 4.6]{DemboOkada}, we have that there is $c>0$ such that 
\begin{align*}
    \limsup_{n \to \infty}\mathbb{E}\left[\exp\left(c \frac{(\log n)^2}{n}(\ca(S[0,n]) - \mathbb{E}[\ca(S[0,n])])\right)\right] < \infty 
\end{align*}
and $\mathbb{E}[\ca(S[0,n])] \sim \frac{\pi^2}{8}\frac{n}{(\log n)}$ (e.g., see \cite[Theorem 1.1]{Chang}). 
Hence, we can derive \eqref{eq:weakercapbound} by applying Markov's inequality.
\end{proof}

In order to discuss the properties of the parts of the random walk inside each annulus, we need to consider the following stopping times. 
Define
\begin{equation*}
\begin{aligned}
&T^i_k:= \inf \{t \ge \tau^i_{k-1}: S_t \le \sqrt{n/K_i} \},\quad 
\tau^i_k:= \inf\{ t \ge T^{i}_k: S_t \ge 2\sqrt{n/K_i}\}.
\end{aligned}
\end{equation*}
These are a sequence of times that indicate when the random walk  hits that ball of radius $\sqrt{n/K_i}$, and later details when it escapes that ball.

The following lemma will ensure that, with high probability, there are not too many times in which the random walk will make an excursion into the ball of radius $\sqrt{n/K_i}$. Furthermore, each excursion into this ball will lie very close to its expected time.
\begin{lem}
Let $M^i= \sup\{k: \tau^i_k \le n  \}$. 
Let $\mathcal{H}^i$ be the event that
$$
\mathcal{H}^i:= \{M^i \le (\log n)^{C} \} \bigcap_{k=1}^{(\log n)^C} \{|\tau^i_k - T^i_k| \le \frac{n}{K_i^{1-\kappa}} \} \bigcap_{k=1}^{(\log n)^C}\{|\tau^i_k - T^i_k| \ge \frac{n}{K_i^{1+\kappa}} \}.
$$
We claim that there is some $C$ such that $\mathbb{P}((\mathcal{H}^i)^c) \ll \frac{1}{(\log n)^2}$.

\end{lem}

\begin{proof}
Notice that $M^i$ can be bounded from above by the number of spheres of radius $\sqrt{n/K_i}$ that are necessary in order to cover the full random walk $S[1,n]$. Considering a sufficiently large $C<\infty$, bounding $M^i$ with high probability is the estimate of \cite[Lemma 3.4]{DemboOkada}. 

Note that the probability that $|\tau^i_k - T^i_k |\ge \frac{n}{K_i^{1-\kappa}} $ is bounded by the probability that a Brownian motion stays within a ball of radius $2 \sqrt{n/K_i}$ for time $\frac{n}{K_i^{1-\kappa}}$. But, this happens with probability less than $\exp[-cK_i^{\kappa}]$. By taking a union bound, this is less than $(\log n)^C \exp[-cK_i^{\kappa}] \ll (\log n)^{-2} $.

Furthermore, the probability that $|\tau^i_k - T^i_k| \le \frac{n}{K_i^{1+\kappa}}$ can be bounded by $\exp[-cK_i^{\kappa}]$ by the reflection principle.  Again, by taking a union bound, this is less than $(\log n)^C \exp[-c K_i^{\kappa}] \ll (\log n)^{-2}$.
\end{proof}

\begin{proof}[Proof of Lemma \ref{lem:pfeqnonintersect}: Estimate for \eqref{eq:capacitydiv-}]

We can also define the event,
\begin{equation*}
\mathcal{G}^i:=\bigcap_{i=1}^{(\log n)^C} \left\{\ca\left(S\left[T^i_k, T^i_k+ \frac{n}{K_i^{1-\kappa}} \right]\right)\le \frac{n}{(\log n)^{1-\chi} K_i^{1-\kappa}} \right\}.
\end{equation*}
By equation \eqref{eq:weakercapbound}, we see that $\mathbb{P}((\mathcal{G}^i)^c) \ll (\log n)^{-2}$.

Finally, we will need the following event,
\begin{equation*}
\mathcal{S}^i:=\left\{ L(\mathcal{B}_{2\sqrt{n/K_i}}) \le \frac{n}{(K_i)^{1-\delta}} \right \}, 
\end{equation*}
where we recall the notation $L(\mathcal{B}_r)$ from Lemma \ref{lem:localtimebnd}. 
By using equation \eqref{eq:momball} and Markov's inequality with $k$ large enough, we would have that
$$
\mathbb{P}((\mathcal{S}^i)^c) \ll (\log n)^{-2}.
$$

Now, we consider the intersection of events 
$\Omega^i := \mathcal{S}^i \cap \mathcal{G}^i \cap \mathcal{H}^i$. 
By combining the probability bounds from earlier, we have $\mathbb{P}((\Omega^i)^c) \ll (\log n)^{-2}$. 
Notice that, on the event $\Omega^i$, we can show that,
\begin{equation*}
\begin{aligned}
&\ca\left(\tilde{S}\left[\frac{n}{(\log n)^4},n\right] \cap \mathcal{A}_{i,n} \right) \le \sum_{k=1}^{M^i} \ca(\tilde{S}[T^i_k, \tau^i_k]) \le \sum_{k=1}^{M^i} \ca(\tilde{S}[T^i_k, T^i_k + \frac{n}{K_i^{1-\kappa}}]) \\ &\le \sum_{k=1}^{M^i} \frac{n}{(\log n)^{1-\chi} K_{i}^{1-\kappa}} 
 \le K_i^{2\kappa}\sum_{k=1}^{M^i} \frac{n}{(\log n)^{1-\chi} K_i^{1+\kappa}} \le \frac{K_i^{2\kappa}}{(\log n)^{1-\chi}}  \sum_{k=1}^{M^i} |\tau^i_k - T^i_k| \\
 &\le \frac{K_i^{2\kappa}}{(\log n)^{1-\chi}} L(\mathcal{B}_{2\sqrt{n/K_i}}) \le \frac{K_i^{2\kappa}}{(\log n)^{1-\chi}}  \frac{n}{(K_i)^{1-\delta}}.
\end{aligned}
\end{equation*}
The second inequality used the fact that $|\tau^i_k - T^i_k| \le \frac{n}{K_i^{1-\kappa}}$ from the event $\mathcal{H}^i$. The third inequality used the bound on the capacity from the event $\mathcal{G}^i$. The fifth inequality used that $|\tau^i_k - T^i_k| \ge \frac{1}{K_i^{1-\kappa}}$ from the event $\mathcal{H}^i$.  Note that the sum of these differences is a lower bound on the number of times in the intersection between $\tilde{S}[\frac{n}{(\log n)^4}, n]$ and $\mathcal{B}_{2\sqrt{n/K_i}}$. The final inequality used the event $\mathcal{S}^i$.

Now, we see that
$$
\begin{aligned}
&\ca\left(\tilde{S}\left[\frac{n}{(\log n)^4},n\right] \cap \mathcal{A}_{i,n} \right)  \max_{z \in \mathcal{A}_{i,n}} G(z) \\&\le \frac{K_i^{2\kappa} n}{(\log n)^{1-\chi} (K_i)^{1-\delta}} \frac{K_{i+1}}{n} \le (\log n)^{3(2\kappa + \delta) +\chi} (\log n)^{-\epsilon} .
\end{aligned}
$$
By setting $11(2\kappa + \delta) + \chi \le \epsilon/2$, we see that the line above will be less than $(\log n)^{-\epsilon/2}$.  We now consider $\hat{\Omega}:= \bigcap_{i=1}^{11} \Omega^i$, we must have $\mathbb{P}(\hat{\Omega}^c) \ll (\log n)^{-2}$. Furthermore, we must also have
$$
\sum_{i=1}^{11}\ca\left(\tilde{S}\left[\frac{n}{(\log n)^4},n\right] \cap \mathcal{A}_{i,n} \right)  \max_{z \in \mathcal{A}_{i,n}} G(z) \le (\log n)^{-\epsilon/2}.
$$

Now, consider the intersection $\tilde{\Omega} \cap \hat{\Omega}$, where we recall $\tilde{\Omega}$ from equation \eqref{eq:deftileomega}. We must have that $\mathbb{P}((\tilde{\Omega} \cap \hat{\Omega})^c) \ll (\log n)^{-2}$ and, in addition, that the right-hand side of equation \eqref{eq:capacitydiv} can be bounded by $(\log n)^{-c}$ for some $c>0$. This completes the proof of Lemma \ref{lem:pfeqnonintersect}. 
\end{proof}

\subsection{An analysis of $X_n - \mathbb{E}[X_n]$}\label{analysis of centered}

In this section, we will prove that $\frac{X_n - \mathbb{E}[X_n]}{\sqrt{\mathrm{Var}(X_n)}} \to 0$ in distribution. We will first start by proving upper and lower bounds on $\mathrm{Var}(X_n)$. Note that an upper bound on $\mathrm{Var}(X_n)$ was already derived in Lemma \ref{lem:uprboundEn}. We now derive the corresponding lower bound for $\mathrm{Var}(X_n)$.
\begin{lem}
Let $X_n$ denote the number of cut points, or the graph distance. We have that
$$
\mathrm{Var}(X_n) \ge \frac{n^2}{(\log n)^{2+o(1)}}.
$$
\end{lem}
\begin{proof}
We perform a dyadic decomposition of $X_n$ up to $(\log n)^{\epsilon}$ levels for some $\epsilon>0$. 
We thus write
$$
X_n  = \sum_{k=0}^{\epsilon\log_2(\log n)} \sum_{l=1}^{2^k} - E_n^{(k,l)} + \sum_{j = 0}^{(\log n)^{\epsilon}-1} X[ j n/(\log n)^{\epsilon}, (j+1)n/ (\log n)^{\epsilon}].
$$

Notice, that since  $\sum_{j = 0}^{(\log n)^{\epsilon}-1} X[ j n/(\log n)^{\epsilon}, (j+1)n/ (\log n)^{\epsilon}]$ is a sum of i.i.d. terms where each individual term has variance at most $\frac{n^2}{(\log n)^{2- o(1)}}$, we can compute its variance as
\begin{equation*}
(\log n)^{\epsilon} \frac{n^2}{(\log n)^{2\epsilon} (\log n)^{2-o(1)}} \le \frac{n^2}{(\log n)^{2 + \epsilon - o(1)}}.
\end{equation*}
This will not be the major contributor to the variance.

We will now consider the contribution to the variance from $\sum_{k=0}^{\epsilon\log_2(\log n)} \sum_{l=1}^{2^k} - E_n^{(k,l)}$. 
Notice, we have,
\begin{equation*}
\begin{aligned}
&\mathrm{Var}\left(\sum_{k=0}^{\epsilon\log_2(\log n)} \sum_{l=1}^{2^k} - E_n^{(k,l)}\right)\\ & = \sum_{k=0}^{\epsilon\log_2(\log n)} \sum_{l=1}^{2^k} \mathrm{Var}(-E_n^{(k,l)}) + \sum_{k_1,k_2,l_1,l_2: (k_1,l_1) \ne (k_2,l_2)} \mathbb{E}\left[E_n^{(k_1,l_1)} E_n^{(k_2,l_2)} - \mathbb{E}[E_n^{(k_1,l_1)}] \mathbb{E}[E_n^{(k_2,l_2)}] \right]\\
& \ge \mathrm{Var}(E_n^{(0,1)}) - \sum_{k_1,k_2,l_1,l_2: (k_1,l_1) \ne (k_2,l_2)} \mathbb{E}[E_n^{(k_1,l_1)}] \mathbb{E}[E_n^{(k_2,l_2)}]\\& \ge \frac{n}{(\log n)^{2+o(1)}} - (\log n)^{2\epsilon}  \frac{n^2}{(\log n)^{3+o(1)}} \gg \frac{n}{(\log n)^{2+o(1)}}.
\end{aligned}
\end{equation*}
We used the fact that $E_n^{(k_1,l_1)}$ are all positive terms, so the only possible negative contribution to the covariance comes from the product of expectations. However, we know from Corollary \ref{coruprboundendime4} that the product of expectations of these expectations is bounded from above by $\frac{n^2}{(\log n)^{3+o(1)}}$. 
This completes the proof of our assertion.
\end{proof}

With a small modification to the proof above, we can derive the following main result.

\begin{thm}
Let $X_n$ denote the number of cut points or the graph distance. We have that, as $n\to \infty$, 
$$
\frac{X_n - \mathbb{E}[X_n]}{\sqrt{\mathrm{Var}(X_n)}} \to 0,
$$
in distribution.
\end{thm}

\begin{proof}
We proceed with the same dyadic decomposition as in the previous proof.

$$
X_n  = \sum_{k=0}^{\epsilon\log_2(\log n)} \sum_{l=1}^{2^k} - E_n^{(k,l)} + \sum_{j = 0}^{(\log n)^{\epsilon}-1} X[ j n/(\log n)^{\epsilon}, (j+1)n/ (\log n)^{\epsilon}].
$$

Let $\Omega^{(k,l)}$ be the set from  Lemma \ref{lem:setOmegahigprob} (constructed for the appropriate cross term $E_n^{(k,l)}$) and consider $\Omega_A:= \bigcap_{(k,l)} \Omega^{(k,l)}$.
By Lemma \ref{lem:setOmegahigprob}, we can ensure that $(\Omega^{(k,l)})^c$ holds with probability no larger than $\text{O}\left( \frac{\log \log n}{\log n} \right)$. By taking a union bound, we see that $\mathbb{P}(\Omega_A^c) = \text{O}\left(\frac{\log \log n}{(\log n)^{1-\epsilon}}\right) \ll 1.$ 
On the event, $\Omega^{(k,l)}$, we see that
$$
| - E_n^{(k,l)}  + \mathbb{E}[E_n^{(k,l)}]| \le  \frac{n}{(\log n)^{3/2+o(1)}}.
$$
Thus, on $\Omega_A$, we must have
$$
\left|\sum_{k=0}^{\epsilon\log_2(\log n)} \sum_{l=1}^{2^k} - E_n^{(k,l)} + \mathbb{E}[E_n^{(k,l)}] \right| \le  (\log n)^{\epsilon}  \frac{n}{(\log n)^{3/2-o(1)}}.
$$
Furthermore, notice that by Lemma \ref{lem:uprboundEn}, 
$$
\begin{aligned}
 \mathrm{Var}\left(\sum_{j = 0}^{(\log n)^{\epsilon}-1} X[ j n/(\log n)^{\epsilon}, (j+1)n/ (\log n)^{\epsilon}] \right) \le \frac{n^2}{(\log n)^{2+ \epsilon +o(1)}}.
\end{aligned}
$$

Let $\Omega_B$ be the event that 
 $\sum_{j = 0}^{(\log n)^{\epsilon}-1} X[ j n/(\log n)^{\epsilon}, (j+1)n/ (\log n)^{\epsilon}] - \mathbb{E}[X[ j n/(\log n)^{\epsilon}, (j+1)n/ (\log n)^{\epsilon}]] $ is smaller than $ \frac{n}{(\log n)^{1+\epsilon/4}}$. By Markov's inequality $\mathbb{P}(\Omega_B^c)\le (\log n)^{-\epsilon/2}$.

Now, on $\Omega_A \cap \Omega_B$, we have
$$
X_n - \mathbb{E}[X_n] \le \frac{n}{(\log n)^{1+\epsilon/4}} + (\log n)^{\epsilon}  \frac{n}{(\log n)^{3/2-o(1)}} \ll \frac{n}{(\log n)^{1+o(1)}} = \sqrt{\mathrm{Var}(X_n)}.
$$
Furthermore, $\mathbb{P}(\Omega_A \cap \Omega_B) = 1-o(1)$ by a union bound. This completes the proof.
\end{proof}

\subsection{An Analysis of Variables with heavy tailed Cross Terms}\label{analysis of ER}

In this subsection, we present, for independent interest, a similar analysis done in the previous sections when we impose weaker conditions on the bounds we obtain for the variance and the expectation for the cross terms under consideration in $d=4$.  Namely, we will assume that the cross terms are very heavy tailed and that we have some almost optimal upper and lower bounds for the variance up to powers of $\log n$. In particular, we will show that the effective resistance in  $d=4$ satisfies these assumptions in Subsection \ref{general:ER:d=4}. Assumption \ref{asmp:weaksmpheavy} will detail what we assume about the cross terms under consideration.

Before this, in Subsection \ref{UB for cross:ER:d=4}, we will briefly derive estimates for the effective resistance, which ensure that  parts (2) and (3) of Assumption  \ref{asmp:weaksmpheavy} will hold. 

\subsubsection{Upper bound for the cross term of the effective resistance}\label{UB for cross:ER:d=4}
We will start by giving the upper bound for the second moment of the cross term of the effective resistance, $E[E_n^2]$. 

\begin{lem} \label{lem:heavytailfour:ER}

Consider $d=4$. 
There exists an event $\Omega$ of probability $O(\frac{\log \log n}{\log n})$ such that 
\begin{equation} \label{eq:esterror:ER}
\begin{aligned}
  E_n \mathbbm{1}[\Omega^c] \le  \mathbb{E}[ E_n] . 
\end{aligned}
\end{equation}
We have that 
\begin{equation} \label{eq:expuprbnd:ER}
\mathbb{E}[ E_n^2] \le \frac{n^2}{(\log n)^{1-o(1)}}.
\end{equation}

\end{lem}
\begin{proof}
Recall $\Omega$ from Lemma \ref{lem:setOmegahigprob}. 
The same proof will allow us to say that on $\Omega^c$, we have that, 
\begin{equation*}
 E_n \le 2 \frac{n}{(\log n)^{b/2}} .
\end{equation*}
Hence, we obtain \eqref{eq:esterror:ER}. 
Notice also that since $ E_n$ is deterministically less than $n$, we must have that,
\begin{equation*}
\mathbb{E}[ E_n^2] = \mathbb{E}[ E_n^2 \mathbbm{1}[\Omega]] + \mathbb{E}[ E_n^2 \mathbbm{1}[\Omega^c]] \le n^2 \mathbb{P}(\Omega) + 4 \frac{n^2}{(\log n)^b} \le \frac{n^2}{(\log n)^{1-o(1)}}.
\end{equation*}
Hence, we obtain \eqref{eq:expuprbnd:ER}. 
\end{proof}

\begin{proof}[Proof of Proposition \ref{prop1+}]
By Lemmas \ref{lem:lowerbndexpdim4} and \ref{lem:heavytailfour:ER}, which also hold for the effective resistance, 
we obtain results  which correspond to parts (2) and (3) of Assumption \ref{asmp:weaksmpheavy} for $b_1=2+o(1)$ and $b_2=1-o(1)$. 
\end{proof}

\subsubsection{General arguments}\label{general:ER:d=4}

\begin{asmp} \label{asmp:weaksmpheavy}
We let $X[a,b]$ for $a, b\in \mathbb{Z}^+, b >a $ be a collection of random variables satisfying the following properties. First, define $E_n:= X_n + X[n,2n] - X_{2n}.$ 
\begin{enumerate}
\item The law of $X[a,b]$ depends only on the value of $|b-a|$.

\item If $[a,b] \cap [c,d] =\emptyset$, then $X[a,b]$ is independent of $X[c,d]$.

\item There exist constants $b_1,b_2$ such that
$$
\frac{n^2}{(\log n)^{b_1}} \le \mathbb{E}[E_n^2] \le \frac{n^2}{(\log n)^{b_2}}.
$$

\item There is an event $\Omega$ with probability $O(\frac{\log \log n}{\log n})$ such that 
$$
E_n \mathbbm{1}[\Omega^c] \le  \mathbb{E}[E_n].
$$

\end{enumerate}
\end{asmp}


We first define the function $f(n)$ as,
\begin{equation*}
\mathrm{Var}(X_n)= \frac{n^2}{(\log n)^{f(n)}}.
\end{equation*}
In order to characterize this, we need to first define the associated function $g(n)$  as follows:
\begin{equation*}
\mathbb{E}[E_{n}^2]:= \frac{n^2}{(\log n)^{g(n)}}. 
\end{equation*}
We remark that the third part of our Assumption \ref{asmp:weaksmpheavy}  shows that we have $b_1+o(1) \le g(n) \le b_2+o(1)$, but we do not get a uniform control of $g(n)$. 
In addition, as a consequence of the fourth part of Assumption \ref{asmp:weaksmpheavy}, we must also have:
\begin{equation} \label{eq:boundenn}
\mathbb{E}[E_{n}^2]  O\left(\frac{\log \log n}{\log n}\right) \ge \mathbb{E}[E_{n}]^2 .
\end{equation}

To try to get a uniform control of $g(n)$ on a fixed number of scales,  we fix a constant $K$ and additionally define the function $g(n,K)$ as:
\begin{equation*}
g(n,K):= \min_{1 \le j \le K} g(n 2^{-j}).
\end{equation*}
This function will be used to characterize the size of the expectation of the square of the graph distance, as detailed in the following lemma.
\begin{lem} \label{lem:freltog}
Fix some $K>0$. Assume that there exists $n$ such that $f(n2^{-K}) \ge g(n,K) + \frac{1}{4}$. Then, we will have that $f(n) \le g(n,K) + o(1) $, where $o(1)$ is a term that will decay to $0$ as $n$ is taken to $\infty$.
\end{lem}

\begin{proof}
We start with a dyadic decomposition of the random variables. 
As a consequence, we observe that,
\begin{equation*}
\begin{aligned}
X_n - \mathbb{E}[X_n] = \sum_{j =1}^{2^K} (\tilde{X}_{K,j} - \mathbb{E}[\tilde{X}_{K,j}])
+\sum_{k=1}^K \sum_{l=1}^{2^k} (E_n^{(k,l)} - \mathbb{E}[E_n^{(k,l)}]).
\end{aligned}
\end{equation*}
We first consider the variance of the term,
\begin{equation*}
\overline{\mathcal{E}}:=-\sum_{k=1}^K \sum_{l=1}^{2^k} (E_n^{(k,l)} - \mathbb{E}[E_n^{(k,l)}]).
\end{equation*}
We remark that when computing the second moment $\mathcal{E}^2$, the only negative contribution can come from terms of the form,
\begin{equation} \label{eq:negcont}
\begin{aligned}
&N_{(k_1,l_1),(k_2,l_2)}:=-\mathbb{E}[(E_n^{(k_1,l_1)} - \mathbb{E}[E_n^{(k_1,l_1)}]) \mathbbm{1}[E_n^{(k_1,l_1)} - \mathbb{E}[E_n^{(k_1,l_1)}] >0]\\&\times (E_n^{(k_2,l_2)} - \mathbb{E}[E_n^{(k_2,l_2)}]) \mathbbm{1}[E_n^{(k_2,l_2)} - \mathbb{E}[E_n^{(k_2,l_2)}]< 0]]\\
&\le \mathbb{E}[(E_n^{(k_1,l_1)} - \mathbb{E}[E_n^{(k_1,l_1)}])^2 \mathbbm{1}[E_n^{(k_1,l_1)} - \mathbb{E}[E_n^{(k_1,l_1)}] >0]^{1/2}\\
& \times \mathbb{E}[(E_n^{(k_2,l_2)} - \mathbb{E}[E_n^{(k_2,l_2)}])^2 \mathbbm{1}[E_n^{(k_2,l_2)} - \mathbb{E}[E_n^{(k_2,l_2)}]< 0]]^{1/2}\\
& \le \mathbb{E}[(E_n^{(k_1,l_1)})^2]^{1/2} \mathbb{E}[E_n^{(k_2,l_2)}] \le \frac{n 2^{-k_1}}{(\log (n 2^{-k_1}))^{g(n2^{-k_1})/2}} \frac{n2^{-k_2} }{(\log n2^{-k_2})^{g(n2^{-{k_2}})/2 +1/2 -o(1) } }\\
&\le \frac{n^2}{ (\log n)^{g(n,K) + 1/2 + o(1) }}.
\end{aligned}
\end{equation}
To get the third inequality, we used the fact that $E_n^{(k_2,l_2)}$ must be positive. Thus, under the condition that $E_n^{(k_2,l_2)} - \mathbb{E}[E_n^{(k_2,l_2)}]$ is negative, the expectation of $\mathbb{E}[(E_n^{(k_2,l_2)} - \mathbb{E}[E_n^{(k_2,l_2)}])^2 \mathbbm{1}[E_n^{(k_2,l_2)} - \mathbb{E}[E_n^{(k_2,l_2)}]< 0]]^{1/2} \le \mathbb{E}[E_n^{(k_2,l_2)}]$. The second-to-last inequality used the definition of $g$ along with the equation \eqref{eq:boundenn} for $n$ large enough. Finally, the last inequality used the fact that $\frac{n}{(\log n)^k}$ is increasing for $n$ large enough, as well as the definition of $g(n,K)$.

This should contrast with the positive contribution coming from,
\begin{equation*}
\mathbb{E}[(E_n^{(k_1,l_1)} - \mathbb{E}[E_n^{(k_1,l_1)}])^2] = \frac{(n2^{-k_1})^2}{(\log n2^{-k_1})^{g(n2^{-k_1})}}.
\end{equation*}

There will be at least one term $(k_1,l_1)$ where $g(n2^{-k_1})$ attains its maximum value, and for this specific value $k^{\max}$, we have that,
\begin{equation*}
\frac{(n2^{-k^{\max}})^2}{(\log n2^{-k^{\max}})^{g(n2^{-k^{\max}})}} = \frac{n^2}{(\log n)^{g(n,K) + o(1)}}.
\end{equation*}

The factor of $o(1)$ is to compensate for possible changes of the constant factor in front. We remark that due to these computations, there is a gap of a factor of $(\log n)^{1/2}$ between the positive term above and the negative contribution from  equation \eqref{eq:negcont}. This will ensure a sufficiently good lower bound for the expectation of $\overline{\mathcal{E}}$. 
Indeed, we have that,
\begin{equation} \label{eq:estcalE}
\begin{aligned}
&\mathbb{E}[\mathcal{E}^2] \ge \sum_{k=1}^K \sum_{l=1}^{2^k} \mathbb{E}[(E_n^{(k_1,l_1)} - \mathbb{E}[E_n^{(k_1,l_1)}])^2] + \sum_{k_1=1}^{K}\sum_{l_1=1}^{2^{k_1}} \sum_{k_2=1}^{K} \sum_{l_1 =1}^{2^{k_2}} N_{(k_1,l_1),(k_2,l_2)} \\
&\ge \frac{n^2}{(\log n)^{g(n,K) +o(1)}} - K^2 4^{K} \frac{n^2}{(\log n)^{g(n,K) + \frac{1}{2} +o(1)}} \ge \frac{n^2}{(\log n)^{g(n,K) + o(1)}}.
\end{aligned}
\end{equation}

Furthermore, if we define
$$
\mathcal{I}:= \sum_{j=1}^{2^K} (\tilde{X}_{(K,j)} - \mathbb{E}[\tilde{X}_{(K,j)}]),
$$
we notice that we have,
\begin{equation} \label{eq:estcalI}
\mathbb{E}[\mathcal{I}^2] \le 2^K \frac{(n2^{-K})^2}{(\log n2^{-K})^{f(n2^{-K})}} \le 2^K \frac{(n2^{-K})^2}{(\log n2^{-K})^{g(n,K) + \frac{1}{4}}} \le \frac{n^2}{(\log n)^{g(n,K) + \frac{1}{4}}}.
\end{equation}
In the first inequality, we used the fact that $\mathcal{I}$ is a sum of independent random variables. Since $g(n,K)$ lies between $b_2+o(1)$ and $b_1+o(1)$, the final inequality must be true for a  large enough $n$.

Finally, notice that, by the Cauchy-Schwarz inequality, 
\begin{equation}\label{var4d:lower}
\mathbb{E}[(\overline{\mathcal{E}} + \mathcal{I})^2] \ge( \mathbb{E}[\overline{\mathcal{E}}^2]^{1/2} - \mathbb{E}[\mathcal{I}^2]^{1/2})^2 \ge \frac{n^2}{(\log n)^{g(n,K) + o(1)}}.
\end{equation}
\end{proof}
 
\begin{rem}
  By \eqref{eq:estcalE}, \eqref{eq:estcalI} and \eqref{var4d:lower}, we have $\epsilon>0$,
\begin{align*}
    \limsup_{n\to \infty} \frac{\mathrm{Var}(X_n)}{n^2 /(\log n)^{b_2+\epsilon}} > 0, \quad 
    \limsup_{n\to \infty} \frac{\mathrm{Var}(X_n)}{n^2/ (\log n)^{b_1-\epsilon}} < \infty
\end{align*}
noting $( \mathbb{E}[\overline{\mathcal{E}}^2]^{1/2} - \mathbb{E}[\mathcal{I}^2]^{1/2})^2\le \mathbb{E}[(\overline{\mathcal{E}} + \mathcal{I})^2] \le( \mathbb{E}[\overline{\mathcal{E}}^2]^{1/2} + \mathbb{E}[\mathcal{I}^2]^{1/2})^2$. 
\end{rem}

We now present the heuristic argument we will use to analyze the behavior of $\frac{X_n - \mathbb{E}[X_n]}{\sqrt{\mathrm{Var}(X_n)}}$.

\begin{itemize}
\item  We start by taking the high probability event $\Omega$ to be the event that all $E_n^{(k,l)} - \mathbb{E}[E_n^{(k,l)}] \le \mathbb{E}[E_n^{(k,l)}] $.

\item If there exists an infinite sequence of $m$ such that $f(n2^{-K})$ is greater than $g(n,K) + \frac{1}{4}$, then  $f(n) \le g(n,K) - o(1)$. For such $n$, we will be able to show that $\frac{X_n - \mathbb{E}[X_n]}{\sqrt{\mathrm{Var}(X_n)}}$ will go to $0$.

\item In the case that the above does not happen, it must be the case that eventually all $f(n2^{-K})  \le g(n,K) + \frac{1}{4} < g(n,K) + \frac{1}{2} $. (The latter condition will ensure that the contribution to the variance of the sum of the graph distance terms will be greater than the contribution of the cross terms.) Then,  if we find an infinite sequence of $n$ such that $f(n)- f(n2^{-K}) \le \frac{1}{\log \log n}$ infinitely often, then we will be able to show for these $n$ that $\frac{X_n - \mathbb{E}[X_n]}{\sqrt{\mathrm{Var}(X_n)}}$ will go to $0$.

\item Finally, we will show that there must be an infinite sequence of $n$ such that $f(n)- f(n2^{-K}) \le \frac{1}{\log \log n}$ infinitely often by contradiction.

\end{itemize}

We can now prove our main result
\begin{thm}
Let $X_n$ be a series of random variables that satisfy the assumptions of \ref{asmp:weaksmpheavy}. 
We have that there exists a subsequence such that as $n\to \infty$, 
\begin{equation*}
\frac{X_n - \mathbb{E}[X_n]}{\sqrt{\mathrm{Var}(X_n)}} \stackrel{d}{\to}0.
\end{equation*}
\end{thm}

\begin{proof}

\textit{Initial Estimates :}
Recall that the function $f(n)$ is defined as $\mathrm{Var}(X_n) = \frac{n^2}{(\log n)^{f(n)}}$. Now, we start by performing the dyadic decomposition to $K$ levels.
Thus, we have,

\begin{equation*}
X_n = \sum_{l=1}^{2^K} \tilde{X}_{(K,l)} - \sum_{k=1}^K \sum_{l=1}^{2^k} E_n^{(k,l)}.
\end{equation*}
Here, $\tilde{X}_{(K,l)}=X[(l-1)n2^{-K}, ln2^{-K}]$ is a collection of independent copies with the same distribution as $X_{n 2^{-K}}$, and for each given $K$, $E_n^{(k,l)}$ is a collection of independent random variables with the same distribution as $E_{n2^{-k}}$.

We now define the event,
\begin{equation*}
\begin{aligned}
\Omega_{(k,l)}:= \left\{ |E_n^{(k,l)} - \mathbb{E}[E_n^{(k,l)}]| \le \mathbb{E}[E_n^{(k,l)}]\right\}, \quad 
\Omega:= \bigcap_{k,l} \Omega_{(k,l)}.
\end{aligned}
\end{equation*}
By the third part of Assumption \ref{asmp:weaksmpheavy}, we see that $\Omega$ occurs with probability $1- o(1)$.
We define $\mathcal{E}= -\sum_{k=1}^K \sum_{l=1}^{2^k} E_n^{(k,l)} $ and observe that,
\begin{equation} \label{eq:varEOm}
\begin{aligned}
& \mathbb{E}[(\mathcal{E} - \mathbb{E}[\mathcal{E}])^2 \mathbbm{1}[\Omega]] \\
\le & \sum_{(k_1,l_1,k_2,l_2)} \mathbb{E}[(E_n^{(k_1,l_1)} - \mathbb{E}[E_n^{(k_1,j_1)}])^{2} \mathbbm{1}[\Omega_{(k_1,l_1)}]]^{1/2} \mathbb{E}[(E_n^{(k_2,j_2)} - \mathbb{E}[E_n^{(k_2,j_2)}])^2 \mathbbm{1}[\Omega_{(k_2,j_2)}]]^{1/2}
\\ \le & 2^{2K+2}  4\max_{(k,l)}( \mathbb{E}[E_n^{(k,l)}]^2) \le\frac{n^2}{(\log n)^{g(n,K) +1 -o(1)}}.
\end{aligned}
\end{equation}
In the last equation, we used equation \eqref{eq:boundenn} to bound $\mathbb{E}[E_n^{(k,l)}]^2$ relative to $\mathbb{E}[(E_n^{(k,l)})^2]$, and we placed irrelevant constant factors into an $o(1)$ exponent of $\log n$ in the denominator. 
In addition, we have that,
\begin{equation} \label{eq:MtOmeg}
\begin{aligned}
&\mathbb{E}\left[\left(\sum_{l=1}^{2^K}  (\tilde{X}_{(K,l)} - \mathbb{E}[\tilde{X}_{(K,l)}])\right)^2 \mathbbm{1}[\Omega]\right] \\ \le & \mathbb{E}\left[\left(\sum_{l=1}^{2^K} (\tilde{X}_{(K,l)} - \mathbb{E}[\tilde{X}_{(K,l)}])\right)^2 \right] 
= \sum_{l=1}^{2^K} \mathbb{E}[(\tilde{X}_{(K,l)} - \mathbb{E}[\tilde{X}_{(K,l)}])^2]\\
 =& 2^{K}\frac{(n 2^{-K})^2}{(\log n2^{-K})^{f(n2^{-K})}} = \frac{n^2}{(\log n)^{f(n)}} 2^{-K}  \frac{(\log n)^{f(n)}}{(\log n2^{-K})^{f(n2^{-K}})}.
\end{aligned}
\end{equation}

\textit{Analysis of properties of $f$}

Now, we seek to analyze various cases. 

\textbf{Case 1: There exists an infinite sequence of $n$ such that $f(n2^{-K}) \ge g(n,K) + \frac{1}{4}$}

In this case,  we observe that Lemma \ref{lem:freltog} would imply that $f(n) \le g(n,K) + o(1)$ along this infinite sequence of $n$. 
Furthermore, we would have that,
\begin{equation*}
\frac{n^2}{(\log n)^{f(n)}} 2^{-K}  \frac{(\log n)^{f(n)}}{(\log n2^{-K})^{f(n2^{-K}})}  \le C \frac{n^2}{(\log n)^{f(n)}} \frac{1}{ (\log n)^{1/4 + o(1)}} \ll \frac{n^2}{(\log n)^{f(n)}},
\end{equation*}
when $n$ is sufficiently large.
Indeed, we used that for $n$ large and fixed $K$, we can bound $\frac{1}{(\log n2^{-K})^{f(n2^{-K})}} \le \frac{1}{ (\log n)^{g(n,K) + \frac{1}{4} +o(1)}}$, where the $o(1)$ term in the exponent  comes from the change $\log n2^{-K} \to \log n$ in the denominator. Recall, we assume $f(n2^{-K}) \ge g(n,K) + \frac{1}{4}.$ In the numerator, we bounded $(\log n)^{f(n)}$ by $(\log n)^{g(n,K)+o(1)}$.  

Combining this estimate with equation \eqref{eq:MtOmeg}, we see that
\begin{equation*}
\mathbb{E}\left[\left(\sum_{l=1}^{2^K} (\tilde{X}_{(K,l)} - \mathbb{E}[\tilde{X}_{(K,l)}]) \right)^2 \mathbbm{1}[\Omega]\right]  \ll \frac{n^2}{(\log n)^{f(n)}}.
\end{equation*}
Thus, there will be some set $\Omega'$, which will still hold with probability $1-o(1)$, such that,
\begin{equation*}
\left(\sum_{l=1}^{2^K} (\tilde{X}_{(K,l)} - \mathbb{E}[\tilde{X}_{(K,l)}]) \right) < \epsilon \sqrt{\mathrm{Var}(X_m)}.
\end{equation*}



Furthermore, recall from equation \eqref{eq:varEOm} that the sums of the cross terms $\mathcal{E}$ satisfy,
\begin{equation*}
\begin{aligned}
&\mathbb{E}[(\mathcal{E} - \mathbb{E}[\mathcal{E}])^2 \mathbbm{1}[\Omega]]  \le \frac{n^2}{(\log n)^{g(n,K) +1 -o(1)}} \ll \frac{n^2}{(\log n)^{f(n)}}.
\end{aligned}
\end{equation*}

Again, there will be a further sub-event $\Omega'' \subset \Omega'$, that will still hold with probability $1- o(1)$, on which we have that $|\mathcal{E}- \mathbb{E}[\mathcal{E}]| < \epsilon \sqrt{\mathrm{Var}(X_n)}$.  Thus, on $\Omega''$, we have $|X_n - \mathbb{E}[X_n]| < 2 \epsilon \sqrt{\mathrm{Var}(X_n)}$. 

By our assumption, we can do this for an infinite sequence of $n$.

\textbf{Case 2: Eventually, all $f(n2^{-K}) \le g(n,K) + \frac{1}{4}$}

This can be divided into two more subcases.

\textbf{Case 2a: There does not exist an infinite sequence of $n$ such that $f(n) - f(n2^{-K}) \le \frac{1}{\log \log n}$}

In this case, we must be able to find a sequence $f(a 2^{mK})$ such that
$$
f(a 2^{mK}) - f(a 2^{(m-1)K}) \ge \frac{1}{\log \log (a 2^{mK})}.
$$
Now, since $\sum_{m=1}^{\infty}\frac{1}{\log \log (a 2^{mK})}$ diverges,
this must mean that $f(a2^{mK})$ is unbounded from above. This contradicts our assumption that $f(n) \le g(n,K) + \frac{1}{4}$ for $n$ large enough. Thus, we do not have to worry about this case.

The only remaining case is the following.

\textbf{Case 2b: There exists an infinite sequence of $n$ such that $f(n) - f(n2^{-K}) \le \frac{1}{\log \log n}$}

In what follows, we will perform computations on values $n$ that satisfy $f(n) - f(n2^{-K}) \le \frac{1}{\log \log n}$.
In this case, we notice that,
\begin{equation*}
\begin{aligned}
&\frac{(\log n)^{f(n)}}{(\log n2^{-K})^{f(n2^{-K}})} \le (\log n)^{f(n) - f(n2^{-K})} \left(\frac{ \log n}{ \log n - K \log 2} \right)^{f(n2^{-K})}\\
\le &(\log n)^{\frac{1}{\log \log n}} \left(\frac{ \log n}{ \log n - K \log 2} \right)^{b_2 + \frac{1}{4}}\le 2e,
\end{aligned}
\end{equation*}
where the final inequality will hold for large enough $n$. Notice that we also used $f(n2^{-K}) \le g(n,K) + \frac{1}{4} \le b_2 + \frac{1}{4} + o(1)$, by our assumption that eventually all $f(n) \le g(n,K) + \frac{1}{4}$ and our upper bound on $g(n,K)$. 

Returning to equation \eqref{eq:MtOmeg}, we can bound the expectation  of the square of $\sum_{l=1}^{2^K} (\tilde{X}_{(K,l)} - \mathbb{E}[\tilde{X}_{(K,l)}])$ by $2e 2^{-K} \mathrm{Var}(X_n)$ for large enough $n$. Thus, by Markov's inequality, we have that,
\begin{equation*}
|\sum_{l=1}^{2^K} (\tilde{X}_{(K,l)} - \mathbb{E}[\tilde{X}_{(K,l)}])| \ge \epsilon \sqrt{\mathrm{Var}(X_n)},
\end{equation*}
with probability at most $2e 2^{-K} \epsilon^{-2}$.

Furthermore, recall again from equation \eqref{eq:varEOm} that we have,
\begin{equation*}
\mathbb{E}[(\mathcal{E} - \mathbb{E}[\mathcal{E}])^2 \mathbbm{1}[\Omega]] \le \frac{n^2}{(\log n)^{g(n,K) +1 -o(1)}} \ll \frac{n^2}{(\log n)^{f(n)}}.
\end{equation*}

Again, here we used that first we are considering an $n$ such that first we have $f(n) \le f(n2^{-K}) + \frac{1}{\log \log n}$. Secondly, under the general assumption of case 2, we have $f(n2^{-K}) + \frac{1}{\log \log n} \le g(n,K) + \frac{1}{4} + \frac{1}{\log \log n} \le g(n,K)+1 + o(1)$ for $n$ large enough. Thus, we have $f(n) \le g(n,K) + 1 + o(1)$ for some infinite sequence of $n$. 
Thus (again by Markov's inequality on the event $\Omega$), with probability $o(1)$, we have  $|\mathcal{E} - \mathbb{E}[\mathcal{E}]| \ge \epsilon \sqrt{\mathrm{Var}(X_n)}$.

In total, we see that with probability at most $o(1) +2e 2^{-K}\epsilon^{-2}$,
\begin{equation*}
|X_n- \mathbb{E}[X_n]| \ge 2 \epsilon \sqrt{\mathrm{Var}(X_n)}.
\end{equation*}

\textit{Finishing the Argument}

Notice that in all cases, for a fixed $K$, we can find a subsequence $X_{n_1},X_{n_2},\ldots$ such that with probability at least $1- 4e 2^{K(1-q)} \epsilon^{-2}$, we must have that
$$
|X_{n_i} - \mathbb{E}[X_{n_i}]| \le \epsilon \sqrt{\mathrm{Var}(X_{n_i})}.
$$
As $K \to \infty$, we see that $1- 4e 2^{-K} \epsilon^{-2} \to 0 $. By a diagonalization argument, we can find a subsequence $\tilde{n}_i$ such that $|X_{\tilde{n}_i} - \mathbb{E}[X_{\tilde{n}_i}]| \le \epsilon \sqrt{\mathrm{Var}(X_{\tilde{n}_i})}$ with probability going to 1 for each $\epsilon$. This subsequence converges in distribution to $0$.
\end{proof}

\end{document}